\documentclass[a4paper,11pt, reqno]{amsart}
\usepackage{latexsym}
\usepackage{amssymb,amsfonts,amsmath,mathrsfs}
\newtheorem{theorem}{Theorem}[section]
\newtheorem{proposition}{Proposition}[section]
\newtheorem{lemma}{Lemma}[section]
\newtheorem{remark}{Remark}[section]

\begin{document} 
\title [Twists of automorphic $L$-functions at the central point]{Twists of automorphic $L$-functions at the central point} 

\author{H. M. Bui}
\address{Institut f\"ur Mathematik, Universit\"at Z\"urich, CH-8057 Z\"urich, Switzerland}
\email{hung.bui@math.uzh.ch}
\thanks{The author is supported by the Swiss National Science Foundation grant 200021\_124737/1.}


\begin{abstract} 
We study the nonvanishing of twists of automorphic $L$-functions at the centre of the critical strip. Given a primitive character $\chi$ modulo $D$ satisfying some technical conditions, we prove that the twisted $L$-functions $L(f.\chi,s)$ do not vanish at $s=\tfrac{1}{2}$ for a positive proportion of primitive forms of weight $2$ and level $q$, for large prime $q$. We also investigate the central values of high derivatives of  $L(f.\chi,s)$, and from that derive an upper bound for the average analytic rank of the studied $L$-functions.
\end{abstract} 

\maketitle 

\section{Introduction}

An important topic in number theory is to understand the behaviour of $L$-functions and their derivatives at the centre of the critical strip (the point of symmetry of the functional equation). One reason for this is the connections with the Birch and Swinnerton-Dyer conjecture and with various deep conjectures on the distribution of zeros of $L$-functions. If there is no trivial reason for an $L$-function to vanish at $s=\tfrac{1}{2}$, for instance because of the sign of the functional equation, the central value is generically expected to be non-zero. Most notably, for quadratic Dirichlet $L$-functions, it is known that at least $7/8$ of quadratic Dirichlet $L$-functions do not vanish at $s=\tfrac{1}{2}$ [\textbf{\ref{S2}}]. There is an extensive literature on the nonvanishing of various families of automorphic $L$-functions. For results involving positive proportions, we refer the readers to [\textbf{\ref{V1}},\textbf{\ref{KM}},\textbf{\ref{IS}},\textbf{\ref{KMV}},\textbf{\ref{KMV1}},\textbf{\ref{MV}},\textbf{\ref{MV1}},\textbf{\ref{Bl}},\textbf{\ref{Kh}}], with no claim of being exhaustive.

In this paper, we study the central values of twists of automorphic $L$-functions. For $q$ prime, we denote by $S_{2}^{*}(q)$ the set of primitive Hecke cuspidal eigenforms of weight 2 relative to the subgroup $\Gamma_0(q)$. Any $f\in S_{2}^{*}(q)$  has a Fourier expansion at infinity
\begin{equation*}
f(z)=\sum_{n=1}^{\infty}n^{1/2}\lambda_f(n)e(nz),
\end{equation*}
normalised so that $\lambda_f(1)=1$. Let
\begin{equation*}
L(f,s)=\sum_{n=1}^{\infty}\frac{\lambda_f(n)}{n^s}
\end{equation*}
be the associated automorphic $L$-function. For any primitive Dirichlet character $\chi$ modulo $D$ with $(q,D)=1$, the twisted $L$-function 
\begin{equation*}
L(f.\chi,s)=\sum_{n=1}^{\infty}\frac{\chi(n)\lambda_f(n)}{n^s}
\end{equation*}
is entire and satisfies a functional equation [\textbf{\ref{L}}]
\begin{eqnarray*}
\Lambda(f.\chi,s)&:=&\hat{q}^s\Gamma\big(s+\tfrac{1}{2}\big)L(f.\chi,s)\\
&=&\varepsilon_{f.\chi}\Lambda(f.\overline{\chi},1-s),
\end{eqnarray*}
where $\hat{q}=\sqrt{q}D/2\pi$ and
\begin{equation}\label{600}
\varepsilon_{f.\chi}=\chi(-q)\sqrt{q}\lambda_{f}(q)\frac{\tau(\chi)}{\tau(\overline{\chi})},
\end{equation}
with $\tau(\chi)$ is the Gauss sum.

In [\textbf{\ref{D}}], Duke showed that for $\chi$ fixed and primitive there exist an absolute constant $c>0$ and a constant $c_D>0$ depending only on $D$ such that for prime $q>c_D$ there are at least $cq(\log q)^{-2}$ forms $f\in S_{2}^{*}(q)$ for which $L(f.\chi,\tfrac{1}{2})\ne0$. We note that 
\begin{equation*}
\big|S_{2}^{*}(q)\big|=\frac{q}{12}+O(1),
\end{equation*}
as $q\rightarrow\infty$. In the case $\chi$ is trivial, Duke's result has been subsequently sharpened by [\textbf{\ref{V1}},\textbf{\ref{KM}},\textbf{\ref{IS}},\textbf{\ref{KMV}}] to give a positive proportion of non-zero central values. These results are obtained by calculating the mollified moments of the family $\big\{L(f,s)\big\}_{f\in S_{2}^{*}(q)}$. Precisely, let $y=(\sqrt{q}/2\pi)^{\Delta}$ for a fixed $0<\Delta<1$, and define the mollifier
\begin{equation}\label{511}
M(f)=\sum_{m\leq y}\frac{x_m\lambda_f(m)}{\sqrt{m}},
\end{equation}
where $X=(x_m)$ is a sequence of real numbers supported on $1\leq m\leq y$ with $x_1=1$ and $x_m\ll1$. The purpose of the function $M(f)$ is to smooth out or ``mollify'' the large values of $L(f,\tfrac{1}{2})$ as we average over $f\in S_{2}^{*}(q)$. By Cauchy's inequality we have
\begin{equation}\label{507}
\sum_{\substack{f\in S_{2}^{*}(q) \\ L(f,\frac{1}{2})\neq 0}}1\ \geq\frac{\big(\sum_{f}L(f,\tfrac{1}{2})M(f)\big)^2}{\sum_{f}\big|L(f,\tfrac{1}{2})M(f)\big|^2}.
\end{equation}
Maximising the ratio on the right hand side with respect to the vector $X=(x_m)$, the optimal coefficients turn out to be
\begin{equation*}
x_m=\mu(m)\prod_{p|m}\bigg(1+\frac{1}{p}\bigg)^{-1}\bigg(1-\frac{\log m}{\log y}\bigg).
\end{equation*}
The optimal proportion obtained from \eqref{507} is $1/4$, which corresponds to the choice $\Delta=1$. In fact, Iwaniec and Sarnak [\textbf{\ref{IS}}] proved a slightly stronger result that at least $1/4$ of these forms satisfying $L(f,\tfrac{1}{2})\geq(\log q)^{-2}$, and moreover, any improvement of that proportion in this context is intimately connected to the Landau-Siegel zeros. We remark that due to the sign of the functional equation for $L(f,s)$, $\varepsilon_{f}=\sqrt{q}\lambda_{f}(q)=\pm1$, $L(f,\tfrac{1}{2})=0$ trivially for asymptotically (as $q\rightarrow\infty$) half of the forms $f\in S_{2}^{*}(q)$, and hence the expected proportion of nonvanishing for $\big\{L(f,\tfrac{1}{2})\big\}_{f\in S_{2}^{*}(q)}$ is 1/2. The same percentage is expected to hold for $\big\{L(f.\chi,\tfrac{1}{2})\big\}_{f\in S_{2}^{*}(q)}$ with a fixed quadratic character $\chi$ (also because $\varepsilon_{f.\chi}=\pm 1$ in \eqref{600}), while with a fixed nonquadratic character $\chi$, it is believed that $L(f.\chi,\tfrac{1}{2})$ does not vanish for all $f\in S_{2}^{*}(q)$.

In the case $\chi$ is fixed, primitive and nonquadratic, taking the corresponding usual mollifier
\begin{equation}\label{508}
\sum_{m\leq y}\frac{\mu(m)\chi(m)\lambda_f(m)}{\sqrt{m}}\prod_{p|m}\bigg(1+\frac{\chi^2(p)}{p}\bigg)^{-1}\bigg(1-\frac{\log m}{\log y}\bigg)
\end{equation}
and proceeding the same as above, one can show that at least $1/3$ of the values $L(f.\chi,\frac{1}{2})$ do not vanish, i.e.
\begin{equation*}
\sum_{\substack{f\in S_{2}^{*}(q) \\ L(f.\chi,\frac{1}{2})\neq 0}}1\ \geq\big(\tfrac{1}{3}+o(1)\big)|S_{2}^{*}(q)|,
\end{equation*}
as $q\rightarrow\infty$. We note that this is the same proportion obtained by Iwaniec and Sarnak in [\textbf{\ref{IS1}}] for primitive Dirichlet $L$-functions. The reason that these proportions are both $1/3$ and not $1/4$  as in the case of automorphic $L$-functions with trivial character in [\textbf{\ref{IS}}] is that the family $\big\{L(f.\chi,s)\big\}_{f\in S_{2}^{*}(q)}$ for fixed primitive nonquadratic $\chi$, and the family $\big\{L(s,\chi)\big\}_{\chi\ \textrm{primitive}\ (\textrm{mod}\ q)}$ are both predicted to have a ``unitary" symmetry (according to the Katz-Sarnak philosophy [\textbf{\ref{KS}}]), while $\big\{L(f,s)\big\}_{f\in S_{2}^{*}(q)}$ is supposed to admit an ``orthogonal" symmetry. In both cases, the results fit well with the predictions of Keating and Snaith [\textbf{\ref{KS1}}] using random matrix models, and those of Conrey and Snaith [\textbf{\ref{CS}}] using the ratios conjectures.

Recently, we gave a modest improvement to Iwaniec and Sarnak's result on the nonvanishing of Dirichlet $L$-functions by using a new two-piece mollifier, i.e. the sum of two mollifiers of different shapes [\textbf{\ref{B}}]. This kind of idea has also been effectively used to show that more than $41\%$ of the zeros of the Riemann zeta-function lying on the critical line [\textbf{\ref{BCY}},\textbf{\ref{F}}]. In this article, we make another use of our mollifier to study the complex twists of automorphic $L$-functions. 

For $k\geq0$, we define the proportion
\begin{equation*}
p_{k,\chi}=\liminf_{q\rightarrow\infty}\frac{\big|\{f\in S_{2}^{*}(q): L^{(k)}(f.\chi,1/2)\ne0\}\big|}{|S_{2}^{*}(q)|}.
\end{equation*}

\begin{theorem}\label{308}
Suppose $\chi$ is a fixed primitive nonquadratic character. Then we have
\begin{equation*}
p_{k,\chi}\geq 1-\frac{1}{16k^2} +O(k^{-4}).
\end{equation*}
In particular
\begin{eqnarray*}
p_{0,\chi}\geq0.3411,\quad p_{1,\chi}\geq0.7553,\quad p_{2,\chi}\geq0.9085\quad\emph{and}\quad p_{3,\chi}\geq0.9643.
\end{eqnarray*}
\end{theorem}

Denote by $r_{f.\chi}$ the ``analytic rank'' of $L(f.\chi,s)$, i.e. the order of vanishing of $L(f.\chi,s)$ at $s=\tfrac{1}{2}$.

\begin{theorem} \label{309}
Suppose $\chi$ is a fixed primitive nonquadratic character. Then we have
\begin{equation*}
\frac{1}{|S_{2}^{*}(q)|}\sum_{f\in S_{2}^{*}(q)}r_{f.\chi}\leq 1.0656+o(1),
\end{equation*}
as $q\rightarrow\infty$.
\end{theorem}

\begin{remark}
\begin{enumerate}
\item \emph{It is possible to consider the family of holomorphic cusp forms of a fixed even weight $k\geq2$ with the level $q$ varies over squarefree positive integers as in [\textbf{\ref{IS}}]. However, we have restricted ourselves to the case $k=2$ and $q$ prime in order to use the Petersson trace formula directly, saving considerable technical considerations.}
\item \emph{We emphasise that our two-piece mollifier is only effective when the character $\chi$ is neither trivial nor quadratic. See the definition of our mollifier in the next section and Remark 2.1.}
\item \emph{We can allow $D$ to tend to infinity sufficiently slowly with $q$. In fact, all of our estimates can be made uniformly in $q$ and $D$ as long as $D\ll(\log q)^{1-\varepsilon}$. This condition arises when applying Lemma 3.3 to derive \eqref{240},\eqref{241} and \eqref{242}. See the footnote in Section 6.}
\item \emph{As $k$ tends to infinity, our proportion $p_{k,\chi}$ approaches 1. This is asymptotically best possible as it is expected that $p_{k,\chi}=1$ for every $k\in\mathbb{N}$.}
\item \emph{With the usual mollifier \eqref{508}, it can be shown that
\begin{eqnarray*}
p_{0,\chi}\geq0.3333,\quad p_{1,\chi}\geq0.7544,\quad p_{2,\chi}\geq0.9083,\quad p_{3,\chi}\geq0.9642
\end{eqnarray*}
and
\begin{equation*}
\frac{1}{|S_{2}^{*}(q)|}\sum_{f\in S_{2}^{*}(q)}r_{f.\chi}\leq 1.0745+o(1),
\end{equation*}
as $q\rightarrow\infty$.}
\end{enumerate}
\end{remark}

As discussed in [\textbf{\ref{BCY}}], it requires a significant amount of work in studying these types of problems using two-piece mollifiers, especially when the second mollifier is much more complicated than the usual one (see the definition in the next section). However, in foreseeing how much improvement can be obtained, one can use some heuristic arguments from the ratios conjectures to express various mollified moments of $L$-functions as certain multiple contour integrals. For a variety of examples of such calculations, see [\textbf{\ref{CS}}].  

\subsection{Notation}

Throughout the paper, we denote $\mathscr{L}=\log \hat{q}$, $y_1=\hat{q}^{\Delta_1}$, $y_2=\hat{q}^{\Delta_2}$, $P[m]=P\big(\tfrac{\log y_1/m}{\log y_1}\big)$ and $Q[m]=Q\big(\tfrac{\log y_2/m}{\log y_2}\big)$, where $P(x)$ and $Q(x)$ are two polynomials satisfying $P(0)=Q(0)=0$ and $P(1)=1$. We define
\begin{equation*}
\zeta_q(s)=\prod_{p\nmid q}\bigg(1-\frac{1}{p^s}\bigg)^{-1}\quad\textrm{and}\quad L_q(s,\chi)=\prod_{p\nmid q}\bigg(1-\frac{\chi(p)}{p^s}\bigg)^{-1}\quad(\sigma>1).
\end{equation*}
We let $\varepsilon>0$ be an arbitrarily small positive number, and can change from time to time.

\section{A two-piece mollifier}

We study a two-piece mollifier of the form $M(f,\chi)=M_1(f,\chi)+M_2(f,\chi)$, where
\begin{equation}\label{509}
M_1(f,\chi)=\sum_{m\leq y_1}\frac{\mu(m)\chi(m)\lambda_f(m)P[m]}{\psi(m)\sqrt{m}}
\end{equation}
and
\begin{equation}\label{520}
M_2(f,\chi)=\frac{1}{\mathscr{L}}\sum_{mn\leq y_2}\frac{\mu(m)(\mu*\log)(n)\chi(m)\overline{\chi}(n)\lambda_f(m)\lambda_f(n)Q[mn]}{\psi(m)\overline{\psi}(n)\sqrt{mn}}.
\end{equation}
Here $\psi(m)=\prod_{p|m}(1+\tfrac{\chi^2(p)}{p})$.

A way to (informally) explain the use of our mollifier is to look for a mollifier for the $k$-th derivative. Consider the functional equation in the asymmetric form
\begin{equation*}
L(f.\chi,s)=X(f.\chi,s)L(f.\overline{\chi},1-s),
\end{equation*}
where
\begin{equation*}
X(f.\chi,s)=\varepsilon_{f.\chi}\hat{q}^{1-2s}\Gamma\big(\tfrac{3}{2}-s\big)/\Gamma\big(s+\tfrac{1}{2}\big).
\end{equation*}
Differentiating both sides yields
\begin{eqnarray*}
L'(f.\chi,s)&=&X'(f.\chi,s)L(f.\overline{\chi},1-s)-X(f.\chi,s)L'(f.\overline{\chi},1-s)\nonumber\\
&=&L(f.\chi,s)\bigg(\frac{X'}{X}(f.\chi,s)-\frac{L'}{L}(f.\overline{\chi},1-s)\bigg).
\end{eqnarray*}
We note that
\begin{equation*}
\frac{X'}{X}(f.\chi,s)=-2\mathscr{L}+O\bigg(\frac{1}{1+|t|}\bigg).
\end{equation*}
So taking differentiation of the above expression $(k-1)$ times and heuristically ignoring various (presumably) lower order terms we have
\begin{equation*}
L^{(k)}(f.\chi,s)=(-1)^kL(f.\chi,s)\bigg((2\mathscr{L})^k+k(2\mathscr{L})^{k-1}\frac{L'}{L}(f.\overline{\chi},1-s)+\ldots\bigg).
\end{equation*}
Hence
\begin{equation*}
\frac{1}{L^{(k)}(f.\chi,s)}=(-1)^k(2\mathscr{L})^{-k}\bigg(\frac{1}{L(f.\chi,s)}-\frac{k}{2\mathscr{L}}\frac{1}{L(f.\chi,s)}\frac{L'}{L}(f.\overline{\chi},1-s)+\ldots\bigg).
\end{equation*}
We note that (informally)
\begin{equation*}
\frac{1}{L(f.\chi,\frac{1}{2})}\approx\sum_{m}\frac{\mu(m)\chi(m)\lambda_f(m)}{\sqrt{m}}
\end{equation*}
and
\begin{equation*}
\frac{1}{L(f.\chi,\frac{1}{2})}\frac{L'}{L}(f.\overline{\chi},\tfrac{1}{2})\approx\sum_{m,n}\frac{\mu(m)(\mu*\log)(n)\chi(m)\overline{\chi}(n)\lambda_f(m)\lambda_f(n)}{\sqrt{mn}}.
\end{equation*}
This suggests that our function $M(f,\chi)=M_1(f,\chi)+M_2(f,\chi)$ mollifies the large values of $L(f.\chi,\tfrac{1}{2})$ and all of its derivatives at the same time.

\begin{remark}
\emph{As mentioned in Remark 1.1, our method only works when $\chi$ is nontrivial and nonquadratic. In the case $\chi$ is real, the multiplicativity property of the coefficients $\lambda_f(n)$ implies that the mollifiers $M_1(f,\chi)$ and $M_2(f,\chi)$, as defined in \eqref{509} and \eqref{520}, essentially have the same shape. Both of them are particular cases of the general mollifier \eqref{511}, and hence the two-piece mollifier does not give any improvement to [\textbf{\ref{IS}},\textbf{\ref{KMV}}].}
\end{remark}

\subsection{Setting up}

Our objects of study are high derivatives of $L$-functions, so it is best to work with shifted moments. We define the harmonic average $\sum_{f}^{h}A_f$ to be
\begin{equation*}
{\sum_{f\in S_{2}^{*}(q)}\!\!\!}^{h}\ A_f:=\sum_{f\in S_{2}^{*}(q)}w_f A_f,
\end{equation*}
where $w_f=1/4\pi(f,f)$, with $(f,g)$ being the Petersson inner product on $\Gamma_0(q))\backslash\mathbb{H}$. The advantage of the weights $w_f$ is to make use of the Petersson formula (see Lemma 3.2), which clearly shows strong cancellations in the average of the product $\lambda_f(m)\lambda_f(n)$ over $f\in S_{2}^{*}(q)$ when $m\ne n$.

Let
\begin{eqnarray*}
&&I_k(\alpha)={\sum_{f\in S_{2}^{*}(q)}\!\!\!\!}^{h}\ L(f.\chi,\tfrac{1}{2}+\alpha)M_k(f,\chi),\\
&&J_k(\alpha,\beta)={\sum_{f\in S_{2}^{*}(q)}\!\!\!\!}^{h}\ L(f.\chi,\tfrac{1}{2}+\alpha)L(f.\overline{\chi},\tfrac{1}{2}+\beta)|M_k(f,\chi)|^2,
\end{eqnarray*}
with $k\in\{1,2\}$, and 
\begin{equation*}
J_3(\alpha,\beta)={\sum_{f\in S_{2}^{*}(q)}\!\!\!\!}^{h}\ L(f.\chi,\tfrac{1}{2}+\alpha)L(f.\overline{\chi},\tfrac{1}{2}+\beta)M_1(f,\chi)\overline{M_2(f,\chi)}.
\end{equation*}
In the subsequent Sections 4--8, we shall prove the following lemmas.

\begin{lemma}\label{310}
Suppose $\Delta_1,\Delta_2<1$. Uniformly for $\alpha\ll \mathscr{L}^{-1}$ we have
\begin{eqnarray*}
I_1(\alpha)=\frac{L(2,\chi^4)}{L(1,\chi^2)}P(1)+O(\mathscr{L}^{-1})
\end{eqnarray*}
and
\begin{eqnarray*}
I_2(\alpha)=\frac{L(2,\chi^4)}{L(1,\chi^2)}\bigg(\Delta_2\int_{0}^{1}y_{2}^{-\alpha(1-x)}Q(x)dx-\tfrac{\Delta_2}{2}Q_1(1)\bigg)+O(\mathscr{L}^{-1}),
\end{eqnarray*}
where
\begin{equation*}
Q_1(x)=\int_{0}^{x}Q(r)dr.
\end{equation*}
\end{lemma}

\begin{lemma}\label{305}
Suppose $\Delta_1<1$. Uniformly for $\alpha,\beta\ll \mathscr{L}^{-1}$ we have
\begin{eqnarray}\label{300}
J_1(\alpha,\beta)&=&\bigg|\frac{L(2,\chi^4)}{L(1,\chi^2)}\bigg|^2\frac{d^2}{dadb}\int_{0}^{1}\int_{0}^{1}y_{1}^{\alpha b+\beta a}(\hat{q}^2y_{1}^{a+b})^{-(\alpha+\beta)t}\nonumber\\
&&\qquad\qquad\big(2\Delta_{1}^{-1}+a+b\big)P(x+a)P(x+b)dxdt\bigg|_{a=b=0}+O(\mathscr{L}^{-1}).
\end{eqnarray}
\end{lemma}

\begin{lemma}\label{306}
Suppose $\Delta_2<\Delta_1<1$. Uniformly for $\alpha,\beta\ll \mathscr{L}^{-1}$ we have
\begin{eqnarray}\label{301}
&&\!\!\!\!\!\!\!\!\!\!J_3(\alpha,\beta)=\bigg|\frac{L(2,\chi^4)}{L(1,\chi^2)}\bigg|^2\frac{d^2}{dadb}\bigg\{\tfrac{\Delta_2}{\Delta_1}\int_{0}^{1}\int_{0}^{1}\int_{0}^{x}y_{1}^{\beta a}y_{2}^{\alpha b-\beta u}(\hat{q}^2y_{1}^{a}y_{2}^{b-u})^{-(\alpha+\beta)t}\big(2+\Delta_1a+\Delta_2(b-u)\big)\nonumber\\
&&\!\!\!\!\!\!\!\qquad\quad P\big(1-\tfrac{\Delta_2(1-x)}{\Delta_1}+a\big)Q(x-u+b)dudxdt-\tfrac{\Delta_2}{2\Delta_1}\int_{0}^{1}\int_{0}^{1}y_{1}^{\beta a}y_{2}^{\alpha b}(\hat{q}^2y_{1}^{a}y_{2}^{b})^{-(\alpha+\beta)t}\nonumber\\
&&\!\!\!\!\!\!\!\!\!\!\qquad\qquad\qquad \big(2+\Delta_1a+\Delta_2b\big)P\big(1-\tfrac{\Delta_2(1-x)}{\Delta_1}+a\big)Q_1(x+b)dxdt\bigg\}\bigg|_{a=b=0}+O(\mathscr{L}^{-1}).
\end{eqnarray}
\end{lemma}

\begin{lemma}\label{307}
Suppose $\Delta_2<1$. Uniformly for $\alpha,\beta\ll \mathscr{L}^{-1}$ we have
\begin{eqnarray}\label{303}
J_{2}(\alpha,\beta)&=&\bigg|\frac{L(2,\chi^4)}{L(1,\chi^2)}\bigg|^2\frac{d^2}{dadb}\bigg\{\tfrac{\Delta_2}{2}\int_{0}^{1}\int_{0}^{1}y_{2}^{\alpha b+\beta a}(\hat{q}^2y_{2}^{a+b})^{-(\alpha+\beta)t}\big(2+\Delta_2(a+b)\big)(1-x)^2\\
&&\qquad\qquad\qquad\qquad\qquad\qquad\qquad Q(x+a)Q(x+b)dxdt\nonumber\\
&&\!\!\!\!\!\!\!\!\!\!\!\!\!\!\!\!\!\!\!\!\!\!\!\!\!\!\!\!\!\!\!\!\!\!\!+\Delta_2\int_{0}^{1}\int_{0}^{1}\int_{0}^{x}\int_{0}^{x}y_{2}^{\alpha b+\beta a-\alpha u-\beta v}(\hat{q}^2y_{2}^{a+b-u-v})^{-(\alpha+\beta)t}\nonumber\\
&&\qquad\qquad\qquad\big(2+\Delta_2(a+b-u-v)\big)Q(x-u+a)Q(x-v+b)dudvdxdt\nonumber\\
&&\!\!\!\!\!\!\!\!\!\!\!\!\!\!\!\!\!\!\!\!\!\!\!\!\!\!\!\!\!\!\!\!\!\!\!-\tfrac{\Delta_2}{2}\int_{0}^{1}\int_{0}^{1}\int_{0}^{x}y_{2}^{\alpha b+\beta a-\alpha u}(\hat{q}^2y_{2}^{a+b-u})^{-(\alpha+\beta)t}\big(2+\Delta_2(a+b-u)\big)Q(x-u+a)Q_1(x+b)dudxdt\nonumber\\
&&\!\!\!\!\!\!\!\!\!\!\!\!\!\!\!\!\!\!\!\!\!\!\!\!\!\!\!\!\!\!\!\!\!\!\!-\tfrac{\Delta_2}{2}\int_{0}^{1}\int_{0}^{1}\int_{0}^{x}y_{2}^{\alpha a+\beta b-\beta u}(\hat{q}^2y_{2}^{a+b-u})^{-(\alpha+\beta)t}\big(2+\Delta_2(a+b-u)\big)Q(x-u+a)Q_1(x+b)dudxdt\nonumber
\end{eqnarray}
\begin{equation*}
+\tfrac{\Delta_2}{4}\int_{0}^{1}\int_{0}^{1}y_{2}^{\alpha b+\beta a}(\hat{q}^2y_{2}^{a+b})^{-(\alpha+\beta)t}\big(2+\Delta_2(a+b)\big)Q_1(x+a)Q_1(x+b)dxdt\bigg\}\bigg|_{a=b=0}+O(\mathscr{L}^{-1}).
\end{equation*}
\end{lemma}

The deductions of Theorem 1.1 and Theorem 1.2 are done in Section 9 and Section 10, respectively.

\section{Various lemmas}

In this section we collect some preliminary results which we need to use later. 

\begin{lemma}\label{101} \emph{(Hecke's recursion formula)}
For $m,n\geq1$ and $f\in S_{2}^{*}(q)$, we have
\begin{displaymath}
\lambda_{f}(m)\lambda_{f}(n)=\sum_{\substack{d|(m,n)\\(d,q)=1}}\lambda_{f}\bigg(\frac{mn}{d^2}\bigg).
\end{displaymath}
\end{lemma}

\begin{lemma}\label{102}
For $m,n\geq1$ we have
\begin{displaymath}
{\sum_{f\in S_{2}^{*}(q)}\!\!\!\!}^{h}\ \lambda_{f}(m)\lambda_{f}(n)=\delta_{m,n}-2\pi\sum_{c\geq1}\frac{S(m,n;cq)}{cq}J_{1}\bigg(\frac{4\pi\sqrt{mn}}{cq}\bigg),
\end{displaymath}
where $\delta_{m,n}$ is the Kronecker symbol, $J_{1}(x)$ is the Bessel function of order $1$ and $S(m,n;c)$ is the Kloosterman sum
\begin{displaymath}
S(m,n;c)=\sum_{a(\emph{mod}\
c)}{\!\!\!\!\!\!}^{\displaystyle{*}}\ e\bigg(\frac{ma+n\overline{a}}{c}\bigg).
\end{displaymath}
Moreover we have the estimate
\begin{equation*}
\sum_{c\geq1}\frac{S(m,n;cq)}{cq}J_{1}\bigg(\frac{4\pi\sqrt{mn}}{cq}\bigg)\ll (m,n,q)^{1/2}(mn)^{1/2}q^{-3/2}.
\end{equation*}
\end{lemma}
\begin{proof}
The first part of the lemma is a special case of the Petersson formula for weight $2$ and prime level $q$. The second part follows easily from the bound $J_1(x)\ll x$ and Weil's bound on Kloosterman sums. 
\end{proof}

The above lemma turns out to be sufficient for the mollified first moments, $I_1(\alpha)$ and $I_2(\alpha)$. For the mollified second moments, we require greater cancellations on averages of Kloosterman sums.

\begin{lemma}
Let $N_1N_2\ll q(\log q)^2$, and $m_1m_2\ll q^{1-\delta}$ for some $\delta>0$. Then we have
\begin{displaymath}
\sum_{\substack{n_1\sim N_1\\n_2\sim N_2}}\sum_{c\geq1}\frac{S(m_1n_1,m_2n_2;cq)}{cq}J_{1}\bigg(\frac{4\pi\sqrt{m_1m_2n_1n_2}}{cq}\bigg)\ll_{\varepsilon,\delta} q^{-1+\varepsilon}\sqrt{m_1m_2N_1N_2}.
\end{displaymath}
\end{lemma}
\begin{proof}
See Lemma 3.3 of [\textbf{\ref{KMV}}] or [\textbf{\ref{V}}].
\end{proof}

\begin{lemma}\label{103}
Let
\begin{equation}\label{9}
V(x)=\frac{1}{2\pi i}\int_{(2)}e^{s^2}x^{-s}\frac{ds}{s}.
\end{equation}
Then for any $B>0$ we have 
\begin{eqnarray*}
L(f.\chi,\tfrac{1}{2}+\alpha)=\sum_{n\geq1}\frac{\chi(n)\lambda_f(n)}{n^{1/2+\alpha}}V\bigg(\frac{n}{\hat{q}^{2+\varepsilon}}\bigg)+O_{\varepsilon,B}(\hat{q}^{-B}).
\end{eqnarray*}
\end{lemma}
\begin{proof}
Consider
\begin{equation*}
A=\frac{1}{2\pi i}\int_{(2)}X^se^{s^2}L(f.\chi,\tfrac{1}{2}+\alpha+s)\frac{ds}{s}.
\end{equation*}
We move the line of integration to $\textrm{Re}(s)=-N$, crossing a simple pole at $s=0$. On the new contour, we use the decay of $e^{s^2}$ and the bound $L(f.\chi,\sigma+it)\ll_\sigma \big(\hat{q}^2(1+|t|)\big)^{1/2-\sigma}$ for $\sigma<0$. In doing so we obtain 
\begin{equation*}
A=L(f.\chi,\tfrac{1}{2}+\alpha)+O_N(X^{-N}\hat{q}^{2N}).
\end{equation*}
We now take $X=\hat{q}^{2+\varepsilon}$, and $N=B/\varepsilon$. Finally expressing the $L$-function in the integral as a Dirichlet series we obtain the lemma.
\end{proof}

\begin{lemma}\label{100}
Let $G(s)=e^{s^2}p(s)$ and $p(s)=\tfrac{(\alpha+\beta)^2-4s^2}{(\alpha+\beta)^2}$. Let
\begin{equation}\label{12}
W_{\alpha,\beta}^{\pm}(x)=\frac{1}{2\pi i}\int_{(2)}G(s)g_{\alpha,\beta}^{\pm}(s)x^{-s}\frac{ds}{s},
\end{equation}
where
\begin{displaymath}
g_{\alpha,\beta}^{+}(s)=\frac{\Gamma(1+\alpha+s)\Gamma(1+\beta+s)}{\Gamma(1+\alpha)\Gamma(1+\beta)}\quad\textrm{and}\quad g_{\alpha,\beta}^{-}(s)=\frac{\Gamma(1-\alpha+s)\Gamma(1-\beta+s)}{\Gamma(1+\alpha)\Gamma(1+\beta)}.
\end{displaymath}
Then we have 
\begin{eqnarray*}
L(f.\chi,\tfrac{1}{2}+\alpha)L(f.\overline{\chi},\tfrac{1}{2}+\beta)&=&\sum_{m,n\geq1}\frac{\chi(m)\overline{\chi}(n)\lambda_f(m)\lambda_f(n)}{m^{1/2+\alpha}n^{1/2+\beta}}W_{\alpha,\beta}^{+}\bigg(\frac{mn}{\hat{q}^2}\bigg)\\
&&\!\!\!\!\!\!\!\!\!\!\!\!\!\!\!\!\!\!+\hat{q}^{-2(\alpha+\beta)}\sum_{m,n\geq1}\frac{\overline{\chi}(m)\chi(n)\lambda_f(m)\lambda_f(n)}{m^{1/2-\alpha}n^{1/2-\beta}}W_{\alpha,\beta}^{-}\bigg( \frac{mn}{\hat{q}^2}\bigg).
\end{eqnarray*}
\end{lemma}
\begin{remark}
\emph{A contour shift to Re$(s)=B$, together with Stirling's formula, gives
\begin{equation*}
V(x),W_{\alpha,\beta}^{\pm}(x)\ll_{B}x^{-B}
\end{equation*}
for any $B>0$.} 
\end{remark}
\begin{remark}
\emph{The purpose of the function $p(s)$ in the above lemma is to cancel the poles of the functions $\zeta\big(1\pm(\alpha+\beta)+2s\big)$ at $s=\mp(\alpha+\beta)/2$ in Sections 6--8. This substantially simplifies our later calculations.} 
\end{remark}
\begin{proof}
Consider the integral
\begin{displaymath}
A_{\alpha,\beta}=\frac{1}{2\pi i}\int_{(2)}G(s)\frac{\Lambda(f.\chi,1/2+\alpha+s)\Lambda(f.\overline{\chi},1/2+\beta+s)}{\Gamma(1+\alpha)\Gamma(1+\beta)}\frac{ds}{s}.
\end{displaymath}
We move the line of integration to $\textrm{Re}(s)=-2$ and use Cauchy's theorem. In doing so we obtain
\begin{displaymath}
A_{\alpha,\beta}=R_0+\frac{1}{2\pi i}\int_{(-2)}G(s)\frac{\Lambda(f.\chi,1/2+\alpha+s)\Lambda(f.\overline{\chi},1/2+\beta+s)}{\Gamma(1+\alpha)\Gamma(1+\beta)}\frac{ds}{s},
\end{displaymath}
where $R_0$ is the term arising from the residue of the integrand at $s=0$. Clearly,  
\begin{displaymath}
R_0=\hat{q}^{1+\alpha+\beta}L(f.\chi,\tfrac{1}{2}+\alpha)L(f.\overline{\chi},\tfrac{1}{2}+\beta).
\end{displaymath}
By the change of variable $s$ to $-s$, and the functional equation, we have
\begin{displaymath}
R_0=A_{\alpha,\beta}+\frac{1}{2\pi i}\int_{(2)}G(s)\frac{\Lambda(f.\overline{\chi},1/2-\alpha+s)\Lambda(f.\chi,1/2-\beta+s)}{\Gamma(1+\alpha)\Gamma(1+\beta)}\frac{ds}{s}.
\end{displaymath}
The lemma now follows by expressing the $\Lambda$-functions as Dirichlet series and then integrating term-by-term.
\end{proof}

\subsection{Mellin transform pairs}

Let $P(x)=\sum_{i}a_ix^i$ and $Q(x)=\sum_{j}b_jx^j$. We note the Mellin transform pairs
\begin{equation}\label{11}
P[h]=\sum_{i}\frac{a_ii!}{(\log y_1)^i}\frac{1}{2\pi i}\int_{(2)}\frac{y_{1}^{w}}{w^{i+1}}h^{-w}dw
\end{equation}
and
\begin{eqnarray}\label{8}
Q[h]=\sum_{j}\frac{b_jj!}{(\log y_2)^j}\frac{1}{2\pi i}\int_{(2)}\frac{y_{2}^{w}}{w^{j+1}}h^{-w}dw.
\end{eqnarray}

\section{Evaluating $I_1(\alpha)$}

In view of Lemma \ref{103} we have
\begin{eqnarray*}
L(f.\chi,\tfrac{1}{2}+\alpha)M_1(f,\chi)=\sum_{m,n}\frac{\mu(m)\chi(mn)}{\psi(m)n^{1/2+\alpha}\sqrt{m}}\lambda_f(m)\lambda_f(n)P[m]V\bigg(\frac{n}{\hat{q}^{2+\varepsilon}}\bigg)+O_{\varepsilon,B}(\hat{q}^{-B+\Delta_1/2+\varepsilon}).
\end{eqnarray*}
The sum $I_1(\alpha)$ can be evaluated using the Petersson formula. For the off-diagonal terms, $m\ne n$, Lemma \ref{102} implies that the total contribution is
\begin{equation*}
\ll_\varepsilon q^{-1/2+\varepsilon}\sum_{m\leq y_1}1\ll_\varepsilon q^{-1/2+\Delta_1/2+\varepsilon}.
\end{equation*} 
The main contribution to $I_1(\alpha)$, which comes from the terms $m=n$, is
\begin{equation*}
I_1'(\alpha)=\sum_{n\geq1}\frac{\mu(n)\chi^2(n)}{\psi(n)n^{1+\alpha}}P[n]V\bigg(\frac{n}{\hat{q}^{2+\varepsilon}}\bigg).
\end{equation*}
Using \eqref{11} and \eqref{9} we can write this as
\begin{eqnarray*}
I_1'(\alpha)&=&\sum_{i}\frac{a_ii!}{(\log y_1)^i}\bigg(\frac{1}{2\pi i}\bigg)^2\int_{(2)}\int_{(2)}e^{s^2}\hat{q}^{(2+\varepsilon)s}y_{1}^{w}\sum_{n\geq1} \frac{\mu(n)\chi^2(n)}{\psi(n)n^{1+\alpha+w+s}}\frac{dw}{w^{i+1}}\frac{ds}{s}\nonumber\\
&=&\sum_{i}\frac{a_ii!}{(\log y_1)^i}\bigg(\frac{1}{2\pi i}\bigg)^2\int_{(2)}\int_{(2)}e^{s^2}\hat{q}^{(2+\varepsilon)s}y_{1}^{w}\frac{A(\alpha,w,s)}{L(1+\alpha+w+s,\chi^2)}\frac{dw}{w^{i+1}}\frac{ds}{s},
\end{eqnarray*}
where
\begin{equation*}
A(\alpha,w,s)=\prod_p\bigg(1-\frac{\chi^2(p)}{p^{1+\alpha+w+s}}\bigg)^{-1}\bigg(1-\frac{\chi^2(p)}{\psi(p)p^{1+\alpha+w+s}}\bigg).
\end{equation*}
We note that $A(\alpha,w,s)$ is absolutely and uniformly convergent in some product of fixed half-planes containing the origin. We first move the $w$-contour to $\textrm{Re}(w)=\delta$, and then move the $s$-contour to $\textrm{Re}(s)=-2\delta/(2+\varepsilon)$, where $\delta>0$ is some fixed small constant such that $A(\alpha,w,s)$ converges absolutely. In doing so we only cross a simple pole at $s=0$. By bounding the integral by absolute values, the contribution along the new line is
\begin{equation*}
\ll_{\delta} \hat{q}^{-2\delta}y_{1}^{\delta}\ll_{\delta} \hat{q}^{-(2-\Delta_1)\delta}.
\end{equation*} 
So
\begin{eqnarray*}
I_1'(\alpha)=\sum_{i}\frac{a_ii!}{(\log y_1)^i}\frac{1}{2\pi i}\int_{(\delta)}y_{1}^{w}\frac{A(\alpha,w,0)}{L(1+\alpha+w,\chi^2)}\frac{dw}{w^{i+1}}+O_{\delta}(\hat{q}^{-(2-\Delta_1)\delta}).
\end{eqnarray*}
It is easy to check that $A(0,0,0)=L(2,\chi^4)$, and hence
\begin{eqnarray*}
I_1(\alpha)&=&\frac{L(2,\chi^4)}{L(1,\chi^2)}\sum_{i}\frac{a_ii!}{(\log y_1)^i}\frac{1}{2\pi i}\int_{(\delta)}y_{1}^{w}\frac{dw}{w^{i+1}}+O(\mathscr{L}^{-1})+O_{\varepsilon,\delta}(\hat{q}^{-(2-\Delta_1)\delta}+q^{-1/2+\Delta_1/2+\varepsilon})\\
&=&\frac{L(2,\chi^4)}{L(1,\chi^2)}P(1)+O(\mathscr{L}^{-1})+O_{\varepsilon,\delta}(\hat{q}^{-(2-\Delta_1)\delta}+q^{-1/2+\Delta_1/2+\varepsilon}).
\end{eqnarray*}

\section{Evaluating $I_2(\alpha)$}

In view of Lemma \ref{103} we have
\begin{eqnarray*}
L(f.\chi,\tfrac{1}{2}+\alpha)M_2(f,\chi)&=&\frac{1}{\mathscr{L}}\sum_{m_1,m_2,n}\frac{\mu(m_1)(\mu*\log)(m_2)\chi(m_1n)\overline{\chi}(m_2)}{\psi(m_1)\overline{\psi}(m_2)n^{1/2+\alpha}\sqrt{m_1m_2}}\\
&&\quad \lambda_f(m_1)\lambda_f(m_2)\lambda_f(n)Q[m_1m_2]V\bigg(\frac{n}{\hat{q}^{2+\varepsilon}}\bigg)+O_{\varepsilon,B}(\hat{q}^{-B+\Delta_2/2+\varepsilon}).
\end{eqnarray*}
Using Lemma \ref{101} and replacing $m_1,n$ by $um_1,un$, the first term is equal to
\begin{eqnarray}\label{514}
\frac{1}{\mathscr{L}}\sum_{(u,q)=1}\frac{\mu(um_1)(\mu*\log)(m_2)\chi^2(u)\chi(m_1n)\overline{\chi}(m_2)}{\psi(um_1)\overline{\psi}(m_2)n^{1/2+\alpha}u^{1+\alpha}\sqrt{m_1m_2}}\lambda_f(m_1n)\lambda_f(m_2)Q[um_1m_2]V\bigg(\frac{un}{\hat{q}^{2+\varepsilon}}\bigg).
\end{eqnarray}

The sum $I_2(\alpha)$ can now be evaluated using the Petersson formula. For the off-diagonal terms, $m_2\ne m_1n$, Lemma \ref{102} implies that the total contribution is
\begin{equation*}
\ll_\varepsilon q^{-1/2+\varepsilon}\sum_{m_1m_2\leq y_2}1\ll_\varepsilon q^{-1/2+\Delta_2/2+\varepsilon}.
\end{equation*} 
The main contribution to $I_2(\alpha)$, which comes from the terms $m_2=m_1n$, is
\begin{equation*}
I_2'(\alpha)=\frac{1}{\mathscr{L}}\sum_{\substack{m_2=m_1n\\(u,q)=1}}\frac{\mu(um_1)(\mu*\log)(m_2)\chi^2(u)}{\psi(um_1)\overline{\psi}(m_2)n^{1/2+\alpha}u^{1+\alpha}\sqrt{m_1m_2}}Q[um_1m_2]V\bigg(\frac{un}{\hat{q}^{2+\varepsilon}}\bigg).
\end{equation*}
Using \eqref{8} and \eqref{9} we can write this as
\begin{eqnarray*}
I_2'(\alpha)&=&\frac{1}{\mathscr{L}}\sum_{j}\frac{b_jj!}{(\log y_2)^j}\bigg(\frac{1}{2\pi i}\bigg)^2\int_{(2)}\int_{(2)}e^{s^2}\hat{q}^{(2+\varepsilon)s}y_{2}^{w}\\
&&\qquad\sum_{\substack{m_2=m_1n\\(u,q)=1}} \frac{\mu(um_1)(\mu*\log)(m_2)\chi^2(u)}{\psi(um_1)\overline{\psi}(m_2)n^{1/2+\alpha}u^{1+\alpha}\sqrt{m_1m_2}}\frac{1}{(um_1m_2)^w}\frac{1}{(un)^s}\frac{dw}{w^{j+1}}\frac{ds}{s}.
\end{eqnarray*}
The sum in the integrand is
\begin{equation*}
-\frac{d}{d\gamma}\sum_{\substack{m_2=m_1n\\(u,q)=1}} \frac{\mu(um_1)\mu(m_{21})\chi^2(u)}{\psi(um_1)\overline{\psi}(m_{21}m_{22})n^{1/2+\alpha}u^{1+\alpha}\sqrt{m_{1}m_{21}m_{22}}m_{22}^{\gamma}}\frac{1}{(um_{1}m_{21}m_{22})^w}\frac{1}{(un)^s}\bigg|_{\gamma=0}.
\end{equation*}
We note that here and throughout the paper, we take $\gamma,\gamma_1,\gamma_2\in\mathbb{C}$ and $\gamma,\gamma_1,\gamma_2\ll \mathscr{L}^{-1}$. Hence 
\begin{equation}\label{23}
I_2'(\alpha)=-\frac{1}{\mathscr{L}}\sum_{j}\frac{b_jj!}{(\log y_2)^j}\frac{\partial}{\partial\gamma}I_2''(\alpha,\gamma)\bigg|_{\gamma=0},
\end{equation}
where
\begin{eqnarray*}
I_2''(\alpha,\gamma)&=&\bigg(\frac{1}{2\pi i}\bigg)^2\int_{(2)}\int_{(2)}e^{s^2}\hat{q}^{(2+\varepsilon)s}y_{2}^{w}\\
&&\!\!\!\!\!\!\!\!\!\!\!\!\!\!\!\!\!\!\!\!\!\!\!\!\!\!\sum_{\substack{m_{21}m_{22}=m_1n\\(u,q)=1}} \frac{\mu(um_1)\mu(m_{21})\chi^2(u)}{\psi(um_{1})\overline{\psi}(m_{21}m_{22})n^{1/2+\alpha}u^{1+\alpha}\sqrt{m_{1}m_{21}m_{22}}m_{22}^{\gamma}}\frac{1}{(um_{1}m_{21}m_{22})^w}\frac{1}{(un)^s}\frac{dw}{w^{j+1}}\frac{ds}{s}.
\end{eqnarray*}
After some standard calculations, the above sum is
\begin{equation}\label{13}
\frac{B(\alpha,\gamma,w,s)\zeta(1+\alpha+\gamma+w+s)\zeta(1+2w)}{L_q(1+\alpha+w+s,\chi^2)\zeta(1+\alpha+w+s)\zeta(1+\gamma+2w)},
\end{equation}
where $B(\alpha,\gamma,w,s)$ is an arithmetical factor given by some Euler product that is absolutely and uniformly convergent in some product of fixed half-planes containing the origin. We first move the $w$-contour to $\textrm{Re}(w)=\delta$, and then move the $s$-contour to $\textrm{Re}(s)=-2\delta/(2+\varepsilon)$, where $\delta>0$ is some fixed small constant such that the arithmetical factor converges absolutely. In doing so we only cross a simple pole at $s=0$. By bounding the integral by absolute values, the contribution along the new line is
\begin{equation*}
\ll_{\delta} \hat{q}^{-2\delta}y_{2}^{\delta}\ll_{\delta} \hat{q}^{-(2-\Delta_2)\delta}.
\end{equation*} 
Thus
\begin{eqnarray}\label{515}
I_2''(\alpha,\gamma)=\frac{1}{2\pi i}\int_{(\delta)}y_{2}^w\frac{B(\alpha,\gamma,w,0)}{L_q(1+\alpha+w,\chi^2)}\frac{\zeta(1+\alpha+\gamma+w)\zeta(1+2w)}{\zeta(1+\alpha+w)\zeta(1+\gamma+2w)}\frac{dw}{w^{j+1}}+O_{\delta}(\hat{q}^{-(2-\Delta_2)\delta}).
\end{eqnarray}
Moving the contour to $\textrm{Re}(w)\asymp \mathscr{L}^{-1}$ and bounding the integral trivially show that $I_2''(\alpha,\gamma)\ll \mathscr{L}^{j}$. Hence
\begin{equation}\label{10}
\frac{\partial}{\partial\gamma}I_2''(\alpha,\gamma)\bigg|_{\gamma=0}=K_{21}(\alpha)+K_{22}(\alpha)+O(\mathscr{L}^{j})+O_{\delta}(\hat{q}^{-(2-\Delta_2)\delta}),
\end{equation}
where
\begin{equation*}
K_{21}(\alpha)=\frac{1}{2\pi i}\int_{(\mathscr{L}^{-1})}y_{2}^{w}\frac{B(\alpha,0,w,0)}{L_q(1+\alpha+w,\chi^2)}\frac{\zeta'(1+\alpha+w)}{\zeta(1+\alpha+w)}\frac{dw}{w^{j+1}}
\end{equation*}
and
\begin{equation*}
K_{22}(\alpha)=-\frac{1}{2\pi i}\int_{(\mathscr{L}^{-1})}y_{2}^{w}\frac{B(\alpha,0,w,0)}{L_q(1+\alpha+w,\chi^2)}\frac{\zeta'(1+2w)}{\zeta(1+2w)}\frac{dw}{w^{j+1}}.
\end{equation*}

By bounding the integrals with absolute values we have $K_{21}(\alpha),K_{22}(\alpha)\ll \mathscr{L}^{j+1}$. Denote by $K_{21}'(\alpha)$, $K_{22}'(\alpha)$ the same integrals as $K_{21}(\alpha)$ and $K_{22}(\alpha)$, respectively, but with $\frac{B(\alpha,0,w,0)}{L_q(1+\alpha+w,\chi^2)}$ being replaced by $\frac{B(\alpha,0,0,0)}{L_q(1+\alpha,\chi^2)}$. Then we have $K_{21}(\alpha)=K_{21}'(\alpha)+O(\mathscr{L}^j)$ and $K_{22}(\alpha)=K_{22}'(\alpha)+O(\mathscr{L}^j)$. The new integrals $K_{21}'(\alpha)$ and $K_{22}'(\alpha)$ have already been evaluated in [\textbf{\ref{B}}] (see Lemma 4.1). From there we obtain
\begin{equation*}
K_{21}(\alpha)=-\frac{B(\alpha,0,0,0)(\log y_2)^{j+1}}{L_q(1+\alpha,\chi^2)j!}\int_{0}^{1}y_{2}^{-\alpha(1-x)}x^{j}dx+O(\mathscr{L}^{j})
\end{equation*}
and
\begin{equation*}
K_{22}(\alpha)=\frac{B(\alpha,0,0,0)(\log y_2)^{j+1}}{2L_q(1+\alpha,\chi^2)(j+1)!}+O(\mathscr{L}^{j}).
\end{equation*}
By \eqref{10} and \eqref{23} we have
\begin{eqnarray*}
I_{2}'(\alpha)=\frac{B(0,0,0,0)}{L(1,\chi^2)}\bigg(\Delta_2\int_{0}^{1}y_{2}^{-\alpha(1-x)}Q(x)dx-\tfrac{\Delta_2}{2}Q_1(1)\bigg)+O(\mathscr{L}^{-1})+O_{\delta}(\hat{q}^{-(2-\Delta_2)\delta}).
\end{eqnarray*}

We now compute $B(0,0,0,0)$. Taking $\alpha=\gamma=0$ and $w=s$ in \eqref{13} we have
\begin{equation*}
B(0,0,s,s)=L_q(1+2s,\chi^2)\sum_{\substack{m_{21}m_{22}=m_1n\\(u,q)=1}}\frac{\mu(um_1)\mu(m_{21})\chi^2(u)}{\psi(um_{1})\overline{\psi}(m_{21}m_{22})u^{1+2s}(m_{1}m_{21}m_{22}n)^{1/2+s}}.
\end{equation*}
Consider the sum over $m_{21}$. The above sum vanishes unless $m_{21}m_{22}=1$. Hence
\begin{eqnarray*}
B(0,0,s,s)&=&L_q(1+2s,\chi^2)\sum_{(u,q)=1}\frac{\mu(u)\chi^2(u)}{\psi(u)u^{1+2s}}\\
&=&L_q(2+4s,\chi^4)\prod_{p\nmid q}\bigg(1+\frac{\chi^2(p)}{p^{1+2s}}\bigg)\bigg(1-\frac{\chi^2(p)}{p^{1+2s}(1+\chi^2(p)/p)}\bigg).
\end{eqnarray*}
Thus $B(0,0,0,0)=(1+O(q^{-1})\big)L(2,\chi^4)$. This and the evaluation of $I_1(\alpha)$ in the previous section complete the proof of Lemma 2.1.

\section{Evaluating $J_1(\alpha,\beta)$}

\subsection{Reduction to a contour integral}

In view of Lemma \ref{100}, we have
\begin{equation*}
L(f.\chi,\tfrac{1}{2}+\alpha)L(f.\overline{\chi},\tfrac{1}{2}+\beta)|M_1(f,\chi)|^2=R_{\alpha,\beta}^{+}(f,\chi)+\hat{q}^{-2(\alpha+\beta)}R_{\alpha,\beta}^{-}(f,\chi),
\end{equation*}
where
\begin{eqnarray*}
R_{\alpha,\beta}^{+}(f,\chi)&=&\sum_{m,n,m_1,n_1}\frac{\mu(m_1)\mu(n_1)\chi(mm_1)\overline{\chi}(nn_1)}{\psi(m_1)\overline{\psi}(n_1)m^{1/2+\alpha}n^{1/2+\beta}\sqrt{m_1n_1}}\\
&&\qquad\lambda_f(m)\lambda_f(n)\lambda_f(m_1)\lambda_f(n_1)P[m_1]P[n_1]W_{\alpha,\beta}^{+}\bigg(\frac{mn}{\hat{q}^2}\bigg)
\end{eqnarray*}
and
\begin{eqnarray*}
R_{\alpha,\beta}^{-}(f,\chi)&=&\sum_{m,n,m_1,n_1}\frac{\mu(m_1)\mu(n_1)\chi(nm_1)\overline{\chi}(mn_1)}{\psi(m_1)\overline{\psi}(n_1)m^{1/2-\alpha}n^{1/2-\beta}\sqrt{m_1n_1}}\\
&&\qquad\lambda_f(m)\lambda_f(n)\lambda_f(m_1)\lambda_f(n_1)P[m_1]P[n_1]W_{\alpha,\beta}^{-}\bigg(\frac{mn}{\hat{q}^2}\bigg).
\end{eqnarray*}
We now consider $\sum^{h}R_{\alpha,\beta}^{+}(f,\chi)$. The sum corresponding to $R_{\alpha,\beta}^{-}(f,\chi)$ can be treated similarly. We wish to use the Petersson formula for the sum over $f$. To do that we first need to appeal to the Hecke formula. From Lemma \ref{101}, replacing $m,n,m_1,n_1$ by $um,vn,vm_1,un_1$ we have
\begin{eqnarray*}
R_{\alpha,\beta}^{+}(f,\chi)&=&\sum_{(uv,q)=1}\frac{\mu(vm_1)\mu(un_1)\chi(mm_1)\overline{\chi}(nn_1)}{\psi(vm_1)\overline{\psi}(un_1)(um)^{1/2+\alpha}(vn)^{1/2+\beta}\sqrt{uvm_1n_1}}\\
&&\qquad\lambda_f(mn_1)\lambda_f(nm_1)P[vm_1]
P[un_1]W_{\alpha,\beta}^{+}\bigg(\frac{uvmn}{\hat{q}^2}\bigg).
\end{eqnarray*}

The sum $\sum^{h}R_{\alpha,\beta}^{+}(f,\chi)$ can now be evaluated using the Petersson formula. For the off-diagonal terms coming from the Kloosterman sums, $mn_1\ne nm_1$, integration by parts and Lemma 3.3 (see [\textbf{\ref{IS}}] or [\textbf{\ref{V}}] for details) imply that the total contribution is\footnote{This is where the condition $D\ll(\log q)^{1-\varepsilon}$ is required.}
\begin{equation}\label{240}
\ll_{\varepsilon} q^{-1+\varepsilon}\sum_{m_1,n_1\leq y_1}1\ll_{\varepsilon} q^{-1+\Delta_1+\varepsilon}.
\end{equation}
The main contribution to $\sum^{h}R_{\alpha,\beta}^{+}(f,\chi)$, which comes from the terms $mn_1=nm_1$, is
\begin{eqnarray*}
J_{1}^{+}(\alpha,\beta)=\sum_{\substack{mn_1=nm_1\\(uv,q)=1}}\frac{\mu(vm_1)\mu(un_1)\chi(mm_1)\overline{\chi}(nn_1)}{\psi(vm_1)\overline{\psi}(un_1)(um)^{1/2+\alpha}(vn)^{1/2+\beta}\sqrt{uvm_1n_1}}P[vm_1]
P[un_1]W_{\alpha,\beta}^{+}\bigg(\frac{uvmn}{\hat{q}^2}\bigg).
\end{eqnarray*}
Using \eqref{11} and \eqref{12} we obtain
\begin{eqnarray*}
&&\!\!\!\!\!\!\!\!\!\!\!\!J_{1}^{+}(\alpha,\beta)=\sum_{i,j}\frac{a_ia_ji!j!}{(\log y_1)^{i+j}}\bigg(\frac{1}{2\pi i}\bigg)^3\int_{(2)}\int_{(2)}\int_{(2)}G(s)g_{\alpha,\beta}^{+}(s)\hat{q}^{2s}y_{1}^{w_1+w_2}\\
&&\!\!\!\!\!\!\!\!\!\!\ \sum_{\substack{mn_1=nm_1\\(uv,q)=1}}\frac{\mu(vm_1)\mu(un_1)\chi(mm_1)\overline{\chi}(nn_1)}{\psi(vm_1)\overline{\psi}(un_1)(um)^{1/2+\alpha}(vn)^{1/2+\beta}\sqrt{uvm_1n_1}}\frac{1}{(vm_1)^{w_1}}\frac{1}{(un_1)^{w_2}}\frac{1}{(uvmn)^s}\frac{dw_1}{w_{1}^{i+1}}\frac{dw_2}{w_{2}^{j+1}}\frac{ds}{s}.
\end{eqnarray*}
The sum in the integrand is
\begin{eqnarray}\label{218}
\frac{C(\alpha,\beta,w_1,w_2,s)\zeta(1+\alpha+\beta+2s)\zeta(1+w_1+w_2)}{L(1+\alpha+w_1+s,\chi^2)L(1+\beta+w_2+s,\overline{\chi}^2)\zeta_q(1+\beta+w_1+s)\zeta_q(1+\alpha+w_2+s)},
\end{eqnarray}
where $C(\alpha,\beta,w_1,w_2,s)$ is an arithmetical factor given by some Euler product that is absolutely and uniformly convergent in some product of fixed half-planes containing the origin. We first move the $w_1$-contour and $w_2$-contour to $\textrm{Re}(w_1)=\textrm{Re}(w_2)=\delta$, and then move the $s$-contour to $\textrm{Re}(s)=-(1-\varepsilon)\delta$, where $\delta,\varepsilon>0$ are some fixed small constants such that the arithmetical factor converges absolutely and $\Delta_1<1-\varepsilon$. In doing so we only cross a simple pole at $s=0$. Note that the simple pole at $s=-(\alpha+\beta)/2$ of $\zeta(1+\alpha+\beta+2s)$ has been cancelled out by the factor $G(s)$. By bounding the integral by absolute values, the contribution along the new line is
\begin{equation*}
\ll_{\varepsilon,\delta} \hat{q}^{-2(1-\varepsilon)\delta}y_{1}^{2\delta}\ll_{\varepsilon,\delta} \hat{q}^{-2(1-\Delta_1-\varepsilon)\delta}.
\end{equation*}
Thus 
\begin{eqnarray*}
J_{1}^{+}(\alpha,\beta)=\zeta(1+\alpha+\beta)\sum_{i,j}\frac{a_ia_ji!j!}{(\log y_1)^{i+j}}L_1(\alpha,\beta)+O_{\varepsilon,\delta}(\hat{q}^{-2(1-\Delta_1-\varepsilon)\delta}),
\end{eqnarray*}
where
\begin{eqnarray*}
L_1(\alpha,\beta)&=&\bigg(\frac{1}{2\pi i}\bigg)^2\int_{(\delta)}\int_{(\delta)}y_{1}^{w_1+w_2}C(\alpha,\beta,w_1,w_2,0)\zeta(1+w_1+w_2)\\
&&\ \frac{1}{L(1+\alpha+w_1,\chi^2)L(1+\beta+w_2,\overline{\chi}^2)}\frac{1}{\zeta_q(1+\beta+w_1)\zeta_q(1+\alpha+w_2)}\frac{dw_1}{dw_{1}^{i+1}}\frac{dw_2}{dw_{2}^{j+1}}.
\end{eqnarray*}

Note that by bounding the integral with absolute values, we get $L_{1}(\alpha,\beta)\ll \mathscr{L}^{i+j-1}$. We denote by $L_{1}'(\alpha,\beta)$ the same integral as $L_{1}(\alpha,\beta)$ but with $L(1+\alpha+w_1,\chi^2)L(1+\beta+w_2,\overline{\chi}^2)$ and $C(\alpha,\beta,w_1,w_2,0)$ being replaced by $L(1+\alpha,\chi^2)L(1+\beta,\overline{\chi}^2)$ and $C(\alpha,\beta,0,0,0)$, respectively. Then we have $L_{1}(\alpha,\beta)=L_{1}'(\alpha,\beta)+O(\mathscr{L}^{i+j-2})$. We will later check that $C(0,0,0,0,0)=\big(1+O(q^{-1})\big)|L(2,\chi^4)|^2$ (see the end of the section), a result we will use freely from now on. The new integral $L_{1}'(\alpha,\beta)$ has already been evaluated in [\textbf{\ref{Y}}] (see Lemma 7). From there we obtain, up to an error term of size $O(\mathscr{L}^{i+j-2})$,
\begin{eqnarray*}
L_{1}(\alpha,\beta)=\frac{|L(2,\chi^4)|^2}{L(1+\alpha,\chi^2)L(1+\beta,\overline{\chi}^2)}\frac{(\log y_1)^{i+j-1}}{i!j!}\frac{d^2}{dadb}\int_{0}^{1}y_{1}^{\alpha b+\beta a}(x+a)^i(x+b)^jdx\bigg|_{a=b=0}.
\end{eqnarray*}
Hence
\begin{eqnarray}\label{400}
J_{1}^{+}(\alpha,\beta)=\bigg|\frac{L(2,\chi^4)}{L(1,\chi^2)}\bigg|^2\frac{1}{\Delta_1\mathscr{L}(\alpha+\beta)}\frac{d^2}{dadb}\int_{0}^{1}y_{1}^{\alpha b+\beta a}P(x+a)P(x+b)dx\bigg|_{a=b=0}+O(\mathscr{L}^{-1}).
\end{eqnarray}

\subsection{Deduction of Lemma \ref{305}}

Next we combine $J_{1}^{+}(\alpha,\beta)$ and $J_{1}^{-}(\alpha,\beta)$. We note that essentially $J_{1}^{-}(\alpha,\beta)=\hat{q}^{-2(\alpha+\beta)}J_{1}^{+}(-\beta,-\alpha)$. Writing
\begin{equation*}
U_1(\alpha,\beta)=\frac{y_{1}^{\alpha b+\beta a}-\hat{q}^{-2(\alpha+\beta)}y_{1}^{-\beta b-\alpha a}}{\alpha+\beta}.
\end{equation*}
Using the integral formula
\begin{equation}\label{224}
\frac{1-z^{-\alpha-\beta}}{\alpha+\beta}=(\log z)\int_{0}^{1}z^{-(\alpha+\beta)t}dt,
\end{equation}
we have
\begin{equation*}
U_1(\alpha,\beta)=\mathscr{L}y_{1}^{\alpha b+\beta a}\big(2+\Delta_1(a+b)\big)\int_{0}^{1}(\hat{q}^2y_{1}^{a+b})^{-(\alpha+\beta)t}dt.
\end{equation*}
In view of \eqref{400} and simplify, we obtain \eqref{300}.

We are left to verify that $C(0,0,0,0,0)=\big(1+O(q^{-1})\big)|L(2,\chi^4)|^2$. From \eqref{218} we get
\begin{eqnarray*}
C(0,0,s,s,s)&=&\big(1+O(q^{-1})\big)L(1+2s,\chi^2)L(1+2s,\overline{\chi}^2)\\
&&\qquad\sum_{\substack{mn_1=nm_1\\(uv,q)=1}}\frac{\mu(vm_1)\mu(un_1)\chi(mm_1)\overline{\chi}(nn_1)}{\psi(vm_1)\overline{\psi}(un_1)(uv)^{1+2s}(mnm_1n_1)^{1/2+s}}.
\end{eqnarray*}
The above sum is
\begin{eqnarray*}
&&\!\!\!\!\!\!\!\!\!\!\!\big(1+O(q^{-1})\big)\prod_{p}\bigg(1-\frac{1}{\psi(p)p^{1+2s}}\bigg)\bigg(1-\frac{1}{\overline{\psi}(p)p^{1+2s}}\bigg)\\
&&\sum_{mn_1=nm_1}\frac{\mu(m_1)\mu(n_1)\chi(mm_1)\overline{\chi}(nn_1)}{\psi(m_1)\overline{\psi}(n_1)(mnm_1n_1)^{1/2+s}}\prod_{p|m_1}\bigg(1-\frac{1}{\psi(p)p^{1+2s}}\bigg)^{-1}\prod_{p|n_1}\bigg(1-\frac{1}{\overline{\psi}(p)p^{1+2s}}\bigg)^{-1}\\
&&\!\!\!\!\!\!\!\!\!\!\!=\big(1+O(q^{-1})\big)\zeta(1+2s)\prod_{p}\bigg(1-\frac{1}{\psi(p)p^{1+2s}}\bigg)\bigg(1-\frac{1}{\overline{\psi}(p)p^{1+2s}}\bigg)\\
&&\qquad\bigg\{1-\bigg(1-\frac{1}{\psi(p)p^{1+2s}}\bigg)^{-1}\frac{\chi^2(p)}{\psi(p)p^{1+2s}}-\bigg(1-\frac{1}{\overline{\psi}(p)p^{1+2s}}\bigg)^{-1}\frac{\overline{\chi}^2(p)}{\overline{\psi}(p)p^{1+2s}}\\
&&\qquad\qquad\qquad\qquad\qquad\qquad+\bigg(1-\frac{1}{\psi(p)p^{1+2s}}\bigg)^{-1}\bigg(1-\frac{1}{\overline{\psi}(p)p^{1+2s}}\bigg)^{-1}\frac{1}{|\psi(p)|^2p^{1+2s}}\bigg\}\\
&&\!\!\!\!\!\!\!\!\!\!\!=\big(1+O(q^{-1})\big)\prod_{p}\bigg(1+\frac{\chi^2(p)}{p}\bigg)^{-1}\bigg(1+\frac{\overline{\chi}^2(p)}{p}\bigg)^{-1}\\
&&\qquad\bigg\{1+\frac{\chi^2(p)+\overline{\chi}^2(p)}{p}\bigg(1-\frac{1}{p^{2s}}\bigg)+\frac{1}{p^2}\bigg(1-\frac{1}{p^{2s}}\bigg)^2\bigg(1-\frac{1}{p^{1+2s}}\bigg)^{-1}\bigg\}.
\end{eqnarray*}
Hence $C(0,0,0,0,0)=\big(1+O(q^{-1})\big)|L(2,\chi^4)|^2$.

\section{Evaluating $J_3(\alpha,\beta)$}

\subsection{Reduction to a contour integral}

In view of Lemma \ref{100}, we have
\begin{equation*}
L(f.\chi,\tfrac{1}{2}+\alpha)L(f.\overline{\chi},\tfrac{1}{2}+\beta)M_1(f,\chi)\overline{M_2(f,\chi)}=S_{\alpha,\beta}^{+}(f,\chi)+\hat{q}^{-2(\alpha+\beta)}S_{\alpha,\beta}^{-}(f,\chi),
\end{equation*}
where
\begin{eqnarray*}
S_{\alpha,\beta}^{+}(f,\chi)&=&\frac{1}{\mathscr{L}}\sum_{m,n,m_1,m_2,n_1}\frac{\mu(m_1)(\mu*\log)(m_2)\mu(n_1)\chi(mm_2n_1)\overline{\chi}(nm_1)}{\overline{\psi}(m_1)\psi(m_2)\psi(n_1)m^{1/2+\alpha}n^{1/2+\beta}\sqrt{m_1m_2n_1}}\\
&&\qquad\lambda_f(m)\lambda_f(n)\lambda_f(m_1)\lambda_f(m_2)\lambda_f(n_1)P[n_1]Q[m_1m_2]W_{\alpha,\beta}^{+}\bigg(\frac{mn}{\hat{q}^2}\bigg)
\end{eqnarray*}
and
\begin{eqnarray*}
S_{\alpha,\beta}^{-}(f,\chi)&=&\frac{1}{\mathscr{L}}\sum_{m,n,m_1,m_2,n_1}\frac{\mu(m_1)(\mu*\log)(m_2)\mu(n_1)\chi(nm_2n_1)\overline{\chi}(mm_1)}{\overline{\psi}(m_1)\psi(m_2)\psi(n_1)m^{1/2-\alpha}n^{1/2-\beta}\sqrt{m_1m_2n_1}}\\
&&\qquad\lambda_f(m)\lambda_f(n)\lambda_f(m_1)\lambda_f(m_2)\lambda_f(n_1)P[n_1]Q[m_1m_2]W_{\alpha,\beta}^{-}\bigg(\frac{mn}{\hat{q}^2}\bigg).
\end{eqnarray*}
We now consider $\sum^{h}S_{\alpha,\beta}^{+}(f,\chi)$. The sum corresponding to $S_{\alpha,\beta}^{-}(f,\chi)$ can be treated similarly. We wish to use the Petersson formula for the sum over $f$. To do that we first need to appeal to the Hecke formula. From Lemma \ref{101}, replacing $m,n,m_1,m_2$ by $um,vn,um_1,vm_2$ we have
\begin{eqnarray*}
S_{\alpha,\beta}^{+}(f,\chi)&=&\frac{1}{\mathscr{L}}\sum_{(uv,q)=1}\frac{\mu(um_1)(\mu*\log)(vm_2)\mu(n_1)\chi(mm_2n_1)\overline{\chi}(nm_1)}{\overline{\psi}(um_1)\psi(vm_2)\psi(n_1)(um)^{1/2+\alpha}(vn)^{1/2+\beta}\sqrt{uvm_1m_2n_1}}\\
&&\qquad\lambda_f(mm_1)\lambda_f(nm_2)\lambda_f(n_1)P[n_1]Q[uvm_1m_2]W_{\alpha,\beta}^{+}\bigg(\frac{uvmn}{\hat{q}^2}\bigg).
\end{eqnarray*}
We next replace $m,m_1,n_1$ by $dm,d_1m_1,dd_1n_1$ and use Lemma \ref{101} once more to obtain
\begin{eqnarray*}
S_{\alpha,\beta}^{+}(f,\chi)&=&\frac{1}{\mathscr{L}}\sum_{(uvdd_1,q)=1}\frac{\mu(ud_1m_1)(\mu*\log)(vm_2)\mu(dd_1n_1)\chi(d^2mm_2n_1)\overline{\chi}(nm_1)}{\overline{\psi}(ud_1m_1)\psi(vm_2)\psi(dd_1n_1)(udm)^{1/2+\alpha}(vn)^{1/2+\beta}\sqrt{uvdd_{1}^2m_1m_2n_1}}\\
&&\qquad\lambda_f(mm_1n_1)\lambda_f(nm_2)P[dd_1n_1]Q[uvd_1m_1m_2]W_{\alpha,\beta}^{+}\bigg(\frac{uvdmn}{\hat{q}^2}\bigg).
\end{eqnarray*}

The sum $\sum^{h}S_{\alpha,\beta}^{+}(f,\chi)$ can now be evaluated using the Petersson formula. For the off-diagonal terms coming from the Kloosterman sums, $mm_1n_1\ne nm_2$, integration by parts and Lemma 3.3 imply that the total contribution is
\begin{equation}\label{241}
\ll_\varepsilon q^{-1+\varepsilon}\sum_{\substack{n_1\leq y_1\\m_1m_2\leq y_2}}1\ll_\varepsilon q^{-1+(\Delta_1+\Delta_2)/2+\varepsilon}.
\end{equation}
The main contribution to $\sum^{h}S_{\alpha,\beta}^{+}(f,\chi)$, which comes from the terms $mm_1n_1=nm_2$, is
\begin{eqnarray*}
J_{3}^{+}(\alpha,\beta)&=&\frac{1}{\mathscr{L}}\sum_{\substack{mm_1n_1=nm_2\\(uvdd_1,q)=1}}\frac{\mu(ud_1m_1)(\mu*\log)(vm_2)\mu(dd_1n_1)\chi(d^2mm_2n_1)\overline{\chi}(nm_1)}{\overline{\psi}(ud_1m_1)\psi(vm_2)\psi(dd_1n_1)(udm)^{1/2+\alpha}(vn)^{1/2+\beta}\sqrt{uvdd_{1}^2m_1m_2n_1}}\\
&&\qquad P[dd_1n_1]Q[uvd_1m_1m_2]W_{\alpha,\beta}^{+}\bigg(\frac{uvdmn}{\hat{q}^2}\bigg).
\end{eqnarray*}
Using \eqref{11}, \eqref{8} and \eqref{12} we get
\begin{eqnarray*}
J_{3}^{+}(\alpha,\beta)&=&\frac{1}{\mathscr{L}}\sum_{i,j}\frac{a_ib_ji!j!}{(\log y_1)^{i}(\log y_2)^j}\bigg(\frac{1}{2\pi i}\bigg)^3\int_{(2)}\int_{(2)}\int_{(2)}G(s)g_{\alpha,\beta}^{+}(s)\hat{q}^{2s}y_{1}^{w_1}y_{2}^{w_2}\\
&&\sum_{\substack{mm_1n_1=nm_2\\(uvdd_1,q)=1}}\frac{\mu(ud_1m_1)(\mu*\log)(vm_2)\mu(dd_1n_1)\chi(d^2mm_2n_1)\overline{\chi}(nm_1)}{\overline{\psi}(ud_1m_1)\psi(vm_2)\psi(dd_1n_1)(udm)^{1/2+\alpha}(vn)^{1/2+\beta}\sqrt{uvdd_{1}^2m_1m_2n_1}}\\
&&\quad\quad\quad\quad\quad\frac{1}{(dd_1n_1)^{w_1}}\frac{1}{(uvd_1m_1m_2)^{w_2}}\frac{1}{(uvdmn)^s}\frac{dw_1}{w_{1}^{i+1}}\frac{dw_2}{w_{2}^{j+1}}\frac{ds}{s}.
\end{eqnarray*}
The sum in the integrand is
\begin{eqnarray*}
&&\!\!\!\!\!\!\!-\frac{d}{d\gamma}\sum_{\substack{mm_1n_1=nm_{21}m_{22}\\(uv_1v_2dd_1,q)=1}}\frac{\mu(ud_1m_1)\mu(v_1m_{21})\mu(dd_1n_1)}{\overline{\psi}(ud_1m_1)\psi(v_1v_2m_{21}m_{22})\psi(dd_1n_1)(udm)^{1/2+\alpha}(v_1v_2n)^{1/2+\beta}}\\
&&\!\!\!\!\!\!\!\qquad\frac{\chi(d^2mm_{21}m_{22}n_1)\overline{\chi}(nm_1)}{\sqrt{uv_1v_2dd_{1}^2m_1m_{21}m_{22}n_1}}\frac{1}{(v_2m_{22})^\gamma}\frac{1}{(dd_1n_1)^{w_1}}\frac{1}{(uv_1v_2d_1m_1m_{21}m_{22})^{w_2}}\frac{1}{(uv_1v_2dmn)^s}\bigg|_{\gamma=0}.
\end{eqnarray*}
Standard calculations show that the above sum is
\begin{eqnarray*}
&&\!\!\!\!\!\!\!\!\!\!\!\!\frac{D(\alpha,\beta,\gamma,w_1,w_2,s)\zeta(1+\alpha+\beta+2s)\zeta(1+2w_2)}{L_q(1+\alpha+w_1+s,\chi^2)\zeta(1+\beta+w_1+s)\zeta(1+\gamma+2w_2)}\frac{\zeta_q(1+\beta+\gamma+w_2+s)\zeta_q(1+w_1+w_2)}{\zeta_q(1+\alpha+w_2+s)\zeta_q(1+\beta+w_2+s)}\nonumber\\
&&\qquad\qquad \times\frac{L(1+\alpha+\gamma+w_2+s,\chi^2)L(1+w_1+w_2,\chi^2)}{L(1+\alpha+w_2+s,\chi^2)L(1+\beta+w_2+s,\overline{\chi}^2)L(1+\gamma+w_1+w_2,\chi^2)},
\end{eqnarray*}
where $D(\alpha,\beta,\gamma,w_1,w_2,s)$ is an arithmetical factor given by some Euler product that is absolutely and uniformly convergent in some product of fixed half-planes containing the origin. We first move the $w_1$-contour and $w_2$-contour to $\textrm{Re}(w_1)=\textrm{Re}(w_2)=\delta$, and then move the $s$-contour to $\textrm{Re}(s)=-(1-\varepsilon)\delta$, where $\delta,\varepsilon>0$ are some fixed small constants such that the arithmetical factor converges absolutely and $\Delta_1<1-\varepsilon$. In doing so we only cross a simple pole at $s=0$. Note that the simple pole at $s=-(\alpha+\beta)/2$ of $\zeta(1+\alpha+\beta+2s)$ has been cancelled out by the factor $G(s)$. By bounding the integral by absolute values, the contribution along the new line is
\begin{equation*}
\ll_{\varepsilon,\delta} \hat{q}^{-2(1-\varepsilon)\delta}(y_{1}y_{2})^{\delta}\ll_{\varepsilon,\delta} \hat{q}^{-(2-\Delta_1-\Delta_2-2\varepsilon)\delta}.
\end{equation*}
Thus 
\begin{equation}\label{17}
J_{3}^{+}(\alpha,\beta)=-\frac{\zeta(1+\alpha+\beta)}{\mathscr{L}}\sum_{i,j}\frac{a_ib_ji!j!}{(\log y_1)^{i}(\log y_2)^j}\frac{\partial}{\partial\gamma}L_3(\alpha,\beta,\gamma)\bigg|_{\gamma=0}+O_{\varepsilon,\delta} (\hat{q}^{-(2-\Delta_1-\Delta_2-2\varepsilon)\delta}),
\end{equation}
where
\begin{eqnarray*}
&&\!\!\!\!\!L_3(\alpha,\beta,\gamma)=\bigg(\frac{1}{2\pi i}\bigg)^2\int_{(\delta)}\int_{(\delta)}y_{1}^{w_1}y_{2}^{w_2}D(\alpha,\beta,\gamma,w_1,w_2,0)\frac{\zeta_q(1+w_1+w_2)L(1+w_1+w_2,\chi^2)}{L(1+\gamma+w_1+w_2,\chi^2)}\\
&&\!\!\!\qquad\qquad\frac{1}{L_q(1+\alpha+w_1,\chi^2)\zeta(1+\beta+w_1)}\frac{L(1+\alpha+\gamma+w_2,\chi^2)}{L(1+\alpha+w_2,\chi^2)L(1+\beta+w_2,\overline{\chi}^2)}\\
&&\!\!\!\qquad\qquad\qquad\qquad\frac{\zeta_q(1+\beta+\gamma+w_2)\zeta(1+2w_2)}{\zeta_q(1+\alpha+w_2)\zeta_q(1+\beta+w_2)\zeta(1+\gamma+2w_2)}\frac{dw_1}{dw_{1}^{i+1}}\frac{dw_2}{dw_{2}^{j+1}}.
\end{eqnarray*}

We now take the derivative with respect to $\gamma$ and set $\gamma=0$. We first note that by moving the contours to $\textrm{Re}(w_1)=\textrm{Re}(w_2)\asymp \mathscr{L}^{-1}$ and bounding the integral with absolute values, we get $L_3(\alpha,\beta,\gamma)\ll \mathscr{L}^{i+j-1}$. Hence
\begin{equation}\label{105}
\frac{\partial}{\partial \gamma}L_3(\alpha,\beta,\gamma)\bigg|_{\gamma=0}=L_{31}(\alpha,\beta)+L_{32}(\alpha,\beta)+O(\mathscr{L}^{i+j-1}),
\end{equation}
where
\begin{eqnarray*}
&&\!\!\!\!\!\!\!\!\!\!\!L_{31}(\alpha,\beta)=\bigg(\frac{1}{2\pi i}\bigg)^2\int_{(\mathscr{L}^{-1})}\int_{(\mathscr{L}^{-1})}y_{1}^{w_1}y_{2}^{w_2}D(\alpha,\beta,0,w_1,w_2,0)\zeta_q(1+w_1+w_2)\\
&&\ \frac{1}{L_q(1+\alpha+w_1,\chi^2)\zeta(1+\beta+w_1)}\frac{\zeta_q'(1+\beta+w_2)}{L(1+\beta+w_2,\overline{\chi}^2)\zeta_q(1+\alpha+w_2)\zeta_q(1+\beta+w_2)}\frac{dw_1}{dw_{1}^{i+1}}\frac{dw_2}{dw_{2}^{j+1}}
\end{eqnarray*}
and
\begin{eqnarray*}
&&\!\!\!\!\!\!\!\!\!\!\!L_{32}(\alpha,\beta)=-\bigg(\frac{1}{2\pi i}\bigg)^2\int_{(\mathscr{L}^{-1})}\int_{(\mathscr{L}^{-1})}y_{1}^{w_1}y_{2}^{w_2}D(\alpha,\beta,0,w_1,w_2,0)\zeta_q(1+w_1+w_2)\\
&&\ \frac{1}{L_q(1+\alpha+w_1,\chi^2)\zeta(1+\beta+w_1)}\frac{\zeta'(1+2w_2)}{L(1+\beta+w_2,\overline{\chi}^2)\zeta_q(1+\alpha+w_2)\zeta(1+2w_2)}\frac{dw_1}{dw_{1}^{i+1}}\frac{dw_2}{dw_{2}^{j+1}}.
\end{eqnarray*}

Note that by bounding the integrals with absolute values, we get $L_{31}(\alpha,\beta),L_{32}(\alpha,\beta)\ll \mathscr{L}^{i+j}$. We denote by $L_{31}'(\alpha,\beta)$, $L_{32}'(\alpha,\beta)$ the same integrals as $L_{31}(\alpha,\beta)$ and $L_{32}(\alpha,\beta)$, respectively, but with $L_q(1+\alpha+w_1,\chi^2)L(1+\beta+w_2,\overline{\chi}^2)$ and $D(\alpha,\beta,0,w_1,w_2,0)$ being replaced by $L_q(1+\alpha,\chi^2)L(1+\beta,\overline{\chi}^2)$ and $D(\alpha,\beta,0,0,0,0)$, respectively. Then we have $L_{31}(\alpha,\beta)=L_{31}'(\alpha,\beta)+O(\mathscr{L}^{i+j-1})$, and $L_{32}(\alpha,\beta)=L_{32}'(\alpha,\beta)+O(\mathscr{L}^{i+j-1})$. As in the previous sections, it is standard to check that $D(0,0,0,0,0,0)=\big(1+O(q^{-1})\big)|L(2,\chi^4)|^2$, a result we will use freely from now on. The new integrals $L_{31}'(\alpha,\beta)$ and $L_{32}'(\alpha,\beta)$ have already been evaluated in [\textbf{\ref{B}}] (see Lemma 5.1). From there we obtain
\begin{eqnarray*}
&&\!\!\!\!\!\!\!\!\!\!\!L_{31}(\alpha,\beta)=-\frac{|L(2,\chi^4)|^2(\log y_1)^{i-1}(\log y_2)^{j+1}}{L_q(1+\alpha,\chi^2)L(1+\beta,\overline{\chi}^2)}\int_{0}^{1}\int_{0}^{x}y_{2}^{-\beta u}\\
&&\ \bigg(\frac{\beta(\log y_1) (1-\tfrac{\Delta_2(1-x)}{\Delta_1})^i}{i!}+\frac{(1-\tfrac{\Delta_2(1-x)}{\Delta_1})^{i-1}}{(i-1)!}\bigg)\bigg(\frac{\alpha(\log y_2)(x-u)^j}{j!}+\frac{(x-u)^{j-1}}{(j-1)!}\bigg)dudx\\
&&\qquad\qquad+O(\mathscr{L}^{i+j-1})+O(\mathscr{L}^{i-1+\varepsilon})\nonumber
\end{eqnarray*}
and
\begin{eqnarray*}
L_{32}(\alpha,\beta)&=&\frac{|L(2,\chi^4)|^2(\log y_1)^{i-1}(\log y_2)^{j+1}}{2L_q(1+\alpha,\chi^2)L(1+\beta,\overline{\chi}^2)}\int_{0}^{1}\bigg(\frac{\beta(\log y_1) (1-\tfrac{\Delta_2(1-x)}{\Delta_1})^i}{i!}+\frac{(1-\tfrac{\Delta_2(1-x)}{\Delta_1})^{i-1}}{(i-1)!}\bigg)\\
&&\qquad\qquad \bigg(\frac{\alpha(\log y_2)x^{j+1}}{(j+1)!}+\frac{x^j}{j!}\bigg)dx+O(\mathscr{L}^{i+j-1})+O(\mathscr{L}^{i-1+\varepsilon}).\nonumber
\end{eqnarray*}
We collect these evaluations, \eqref{105}, \eqref{17} and write $J_{3}^{+}(\alpha,\beta)$ in a compact form as
\begin{eqnarray}\label{401}
&&\!\!\!\!\!\!J_{3}^{+}(\alpha,\beta)=\bigg|\frac{L(2,\chi^4)}{L(1,\chi^2)}\bigg|^2\frac{1}{\mathscr{L}(\alpha+\beta)}\frac{d^2}{dadb}\bigg\{\tfrac{\Delta_2}{\Delta_1}\int_{0}^{1}\int_{0}^{x}y_{1}^{\beta a}y_{2}^{\alpha b-\beta u}P\big(1-\tfrac{\Delta_2(1-x)}{\Delta_1}+a\big)\\
&&\quad Q(x-u+b)dudx-\tfrac{\Delta_2}{2\Delta_1}\int_{0}^{1}y_{1}^{\beta a}y_{2}^{\alpha b}P\big(1-\tfrac{\Delta_2(1-x)}{\Delta_1}+a\big)Q_1(x+b)dx\bigg\}\bigg|_{a=b=0}+O(\mathscr{L}^{-1}).\nonumber
\end{eqnarray}

\subsection{Deduction of Lemma \ref{306}}

Next we combine $J_{3}^{+}(\alpha,\beta)$ and $J_{3}^{-}(\alpha,\beta)$. We note that essentially $J_{3}^{-}(\alpha,\beta)=\hat{q}^{-2(\alpha+\beta)}J_{3}^{+}(-\beta,-\alpha)$. Writing
\begin{equation*}
U_3(\alpha,\beta;u)=\frac{y_{1}^{\beta a}y_{2}^{\alpha b-\beta u}-\hat{q}^{-2(\alpha+\beta)}y_{1}^{-\alpha a}y_{2}^{-\beta b+\alpha u}}{\alpha+\beta}.
\end{equation*}
Using \eqref{224} we have
\begin{equation*}
U_3(\alpha,\beta;u)=\mathscr{L}y_{1}^{\beta a}y_{2}^{\alpha b-\beta u}\big(2+\Delta_1a+\Delta_2(b-u)\big)\int_{0}^{1}(\hat{q}^2y_{1}^{a}y_{2}^{b-u})^{-(\alpha+\beta)t}dt.
\end{equation*}
In view of \eqref{401} and simplify, we obtain \eqref{301}.

\section{Evaluating $J_2(\alpha,\beta)$}

\subsection{Reduction to a contour integral}

In view of Lemma \ref{100}, we have
\begin{equation*}
L(f.\chi,\tfrac{1}{2}+\alpha)L(f.\overline{\chi},\tfrac{1}{2}+\beta)|M_2(f,\chi)|^2=T_{\alpha,\beta}^{+}(f,\chi)+\hat{q}^{-2(\alpha+\beta)}T_{\alpha,\beta}^{-}(f,\chi),
\end{equation*}
where
\begin{eqnarray*}
T_{\alpha,\beta}^{+}(f,\chi)&=&\frac{1}{\mathscr{L}^2}\sum_{m,n,m_1,m_2,n_1,n_2}\frac{\mu(m_1)(\mu*\log)(m_2)\mu(n_1)(\mu*\log)(n_2)\chi(mm_1n_2)\overline{\chi}(nm_2n_1)}{\psi(m_1)\overline{\psi}(m_2)\overline{\psi}(n_1)\psi(n_2)m^{1/2+\alpha}n^{1/2+\beta}\sqrt{m_1m_2n_1n_2}}\\
&&\qquad\lambda_f(m)\lambda_f(n)\lambda_f(m_1)\lambda_f(m_2)\lambda_f(n_1)\lambda_f(n_2)Q[m_1m_2]Q[n_1n_2]W_{\alpha,\beta}^{+}\bigg(\frac{mn}{\hat{q}^2}\bigg)
\end{eqnarray*}
and
\begin{eqnarray*}
T_{\alpha,\beta}^{-}(f,\chi)&=&\frac{1}{\mathscr{L}^2}\sum_{m,n,m_1,m_2,n_1,n_2}\frac{\mu(m_1)(\mu*\log)(m_2)\mu(n_1)(\mu*\log)(n_2)\chi(nm_1n_2)\overline{\chi}(mm_2n_1)}{\psi(m_1)\overline{\psi}(m_2)\overline{\psi}(n_1)\psi(n_2)m^{1/2-\alpha}n^{1/2-\beta}\sqrt{m_1m_2n_1n_2}}\\
&&\qquad\lambda_f(m)\lambda_f(n)\lambda_f(m_1)\lambda_f(m_2)\lambda_f(n_1)\lambda_f(n_2)Q[m_1m_2]Q[n_1n_2]W_{\alpha,\beta}^{-}\bigg(\frac{mn}{\hat{q}^2}\bigg).
\end{eqnarray*}
We now consider $\sum^{h}T_{\alpha,\beta}^{+}(f,\chi)$. The sum corresponding to $T_{\alpha,\beta}^{-}(f,\chi)$ can be treated similarly. We wish to use the Petersson formula for the sum over $f$. To do that we first need to appeal to the Hecke formula. From Lemma \ref{101} we write
\begin{eqnarray*}
&&\sum_{m_1,m_2}\frac{\mu(m_1)(\mu*\log)(m_2)\chi(m_1)\overline{\chi}(m_2)\lambda_f(m_1)\lambda_f(m_2)Q[m_1m_2]}{\psi(m_1)\overline{\psi}(m_2)\sqrt{m_1m_2}}\\
&&\qquad\qquad=\sum_{\substack{m_1,m_2\\(u,q)=1}}\frac{\mu(um_1)(\mu*\log)(um_2)\chi(m_1)\overline{\chi}(m_2)\lambda_f(m_1m_2)Q[u^2m_1m_2]}{\psi(um_1)\overline{\psi}(um_2)u\sqrt{m_1m_2}}
\end{eqnarray*}
and
\begin{eqnarray*}
&&\sum_{n_1,n_2}\frac{\mu(n_1)(\mu*\log)(n_2)\chi(n_2)\overline{\chi}(n_1)\lambda_f(n_1)\lambda_f(n_2)Q[n_1n_2]}{\overline{\psi}(n_1)\psi(n_2)\sqrt{n_1n_2}}\\
&&\qquad\qquad=\sum_{\substack{n_1,n_2\\(v,q)=1}}\frac{\mu(vn_1)(\mu*\log)(vn_2)\chi(n_2)\overline{\chi}(n_1)\lambda_f(n_1n_2)Q[v^2n_1n_2]}{\overline{\psi}(vn_1)\psi(vn_2)v\sqrt{n_1n_2}}.
\end{eqnarray*}
Next we consider the factors $\lambda_f(m)\lambda_f(m_1m_2)$ and $\lambda_f(n)\lambda_f(n_1n_2)$. Again using Lemma \ref{101} and the substitutions $m\rightarrow d_1d_2m$, $m_1\rightarrow d_1m_1$, $m_2\rightarrow d_2m_2$, $n\rightarrow d_3d_4n$, $n_1\rightarrow d_3n_1$ and $n_2\rightarrow d_4n_2$, we obtain
\begin{eqnarray*}
&&\!\!\!\!\!\!\!\!\!\!\!\!T_{\alpha,\beta}^{+}(f,\chi)=\frac{1}{\mathscr{L}^2}\sum_{(uvd_1d_2d_3d_4,q)=1}\frac{\mu(ud_1m_1)(\mu*\log)(ud_2m_2)\mu(vd_3n_1)(\mu*\log)(vd_4n_2)}{\psi(ud_1m_1)\overline{\psi}(ud_2m_2)\overline{\psi}(vd_3n_1)\psi(vd_4n_2)(d_1d_2m)^{1/2+\alpha}(d_3d_4n)^{1/2+\beta}}\\
&&\qquad\qquad\qquad\frac{\chi(d_{1}^{2}mm_1n_2)\overline{\chi}(d_{3}^{2}nm_2n_1)}{uv\sqrt{d_1d_2d_3d_4m_1m_2n_1n_2}}\lambda_f(mm_1m_2)\lambda_f(nn_1n_2)\\
&&\qquad\qquad\qquad\qquad\qquad\qquad\qquad Q[u^2d_1d_2m_1m_2]Q[v^2d_3d_4n_1n_2]W_{\alpha,\beta}^{+}\bigg(\frac{d_1d_2d_3d_4mn}{\hat{q}^2}\bigg).
\end{eqnarray*}

The sum $\sum^{h}T_{\alpha,\beta}^{+}(f,\chi)$ can now be evaluated using the Petersson formula. For the off-diagonal terms coming from the Kloosterman sums, $mm_1m_2\ne nn_1n_2$, integration by parts and Lemma 3.3 imply that the total contribution is
\begin{equation}\label{242}
\ll_\varepsilon q^{-1+\varepsilon}\sum_{m_1m_2,n_1n_2\leq y_2}1\ll_\varepsilon q^{-1+\Delta_2+\varepsilon}.
\end{equation}
The main contribution to $\sum^{h}T_{\alpha,\beta}^{+}(f,\chi)$, which comes from the terms $mm_1m_2=nn_1n_2$, is
\begin{eqnarray*}
J_{2}^{+}(\alpha,\beta)&=&\frac{1}{\mathscr{L}^2}\sum_{\substack{mm_1m_2=nn_1n_2\\(uvd_1d_2d_3d_4,q)=1}}\frac{\mu(ud_1m_1)(\mu*\log)(ud_2m_2)\mu(vd_3n_1)(\mu*\log)(vd_4n_2)}{\psi(ud_1m_1)\overline{\psi}(ud_2m_2)\overline{\psi}(vd_3n_1)\psi(vd_4n_2)(d_1d_2m)^{1/2+\alpha}(d_3d_4n)^{1/2+\beta}}\\
&&\qquad \frac{\chi(d_{1}^{2}mm_1n_2)\overline{\chi}(d_{3}^{2}nm_2n_1)}{uv\sqrt{d_1d_2d_3d_4m_1m_2n_1n_2}}Q[u^2d_1d_2m_1m_2]Q[v^2d_3d_4n_1n_2]W_{\alpha,\beta}^{+}\bigg(\frac{d_1d_2d_3d_4mn}{\hat{q}^2}\bigg).
\end{eqnarray*}
Using \eqref{8} and \eqref{12} we obtain
\begin{eqnarray*}
J_{2}^{+}(\alpha,\beta)&=&\frac{1}{\mathscr{L}^2}\sum_{i,j}\frac{b_ib_ji!j!}{(\log y_2)^{i+j}}\bigg(\frac{1}{2\pi i}\bigg)^3\int_{(2)}\int_{(2)}\int_{(2)}G(s)g_{\alpha,\beta}^{+}(s)\hat{q}^{2s}y_{2}^{w_1+w_2}\\
&&\!\!\!\!\!\!\!\!\!\!\!\!\!\!\!\!\!\!\!\!\!\!\!\!\!\!\!\!\!\!\sum_{\substack{mm_1m_2=nn_1n_2\\(uvd_1d_2d_3d_4,q)=1}}\frac{\mu(ud_1m_1)(\mu*\log)(ud_2m_2)\mu(vd_3n_1)(\mu*\log)(vd_4n_2)}{\psi(ud_1m_1)\overline{\psi}(ud_2m_2)\overline{\psi}(vd_3n_1)\psi(vd_4n_2)(d_1d_2m)^{1/2+\alpha}(d_3d_4n)^{1/2+\beta}}\\
&&\!\!\!\!\!\!\!\!\!\!\!\!\!\!\!\!\!\!\!\!\!\!\frac{\chi(d_{1}^{2}mm_1n_2)\overline{\chi}(d_{3}^{2}nm_2n_1)}{uv\sqrt{d_1d_2d_3d_4m_1m_2n_1n_2}}\frac{1}{(u^2d_1d_2m_1m_2)^{w_1}}\frac{1}{(v^2d_3d_4n_1n_2)^{w_2}}\frac{1}{(d_1d_2d_3d_4mn)^s}\frac{dw_1}{w_{1}^{i+1}}\frac{dw_2}{w_{2}^{j+1}}\frac{ds}{s}.
\end{eqnarray*}
The sum in the integrand is
\begin{eqnarray*}
&&\!\!\!\!\!\!\!\!\!\!\frac{d^2}{d\gamma_1d\gamma_2}\sum_{\substack{mm_1m_{21}m_{22}=nn_1n_{21}n_{22}\\(u_1u_2v_1v_2d_1d_{21}d_{22}d_3d_{41}d_{42},q)=1}} \frac{\mu(u_1u_2d_1m_1)\mu(u_1d_{21}m_{21})}{\psi(u_1u_2d_1m_1)\overline{\psi}(u_1u_2d_{21}d_{22}m_{21}m_{22})}\nonumber\\
&&\qquad\qquad\qquad\qquad\qquad\qquad\qquad\qquad\quad\frac{\mu(v_1v_2d_3n_1)\mu(v_1d_{41}n_{21})}{\overline{\psi}(v_1v_2d_3n_1)\psi(v_1v_2d_{41}d_{42}n_{21}n_{22})}\\
&&\frac{\chi(d_{1}^{2}mm_1n_{21}n_{22})\overline{\chi}(d_{3}^{2}nm_{21}n_{22}n_1)}{(d_1d_{21}d_{22}m)^{1/2+\alpha}(d_3d_{41}d_{42}n)^{1/2+\beta}u_1u_2v_1v_2\sqrt{d_1d_{21}d_{22}d_3d_{41}d_{42}m_1m_{21}m_{22}n_1n_{21}n_{22}}}\nonumber\\
&&\qquad\qquad\qquad\frac{1}{(u_2d_{22}m_{22})^{\gamma_1}(v_2d_{42}n_{22})^{\gamma_2}}\frac{1}{(u_{1}^{2}u_{2}^{2}d_1d_{21}d_{22}m_1m_{21}m_{22})^{w_1}}\\
&&\qquad\qquad\qquad\qquad\qquad\frac{1}{(v_{1}^{2}v_{2}^{2}d_3d_{41}d_{42}n_1n_{21}n_{22})^{w_2}}\frac{1}{(d_1d_{21}d_{22}d_3d_{41}d_{42}mn)^s}\bigg|_{\gamma_1=\gamma_2=0}.
\end{eqnarray*}
As in the previous sections, up to an arithmetical factor $E(\alpha,\beta,\gamma_1,\gamma_2,w_1,w_2,s)$, the above sum is
\begin{eqnarray*}
&&\!\!\!\!\!\!\!\!\!\!\frac{\zeta(1+\alpha+\beta+2s)L(1+w_1+w_2,\chi^2)L(1+w_1+w_2,\overline{\chi}^2)\zeta(1+\gamma_1+\gamma_2+w_1+w_2)\zeta^2(1+w_1+w_2)}{L(1+\gamma_1+w_1+w_2,\overline{\chi}^2)L(1+\gamma_2+w_1+w_2,\chi^2)\zeta(1+\gamma_1+w_1+w_2)\zeta(1+\gamma_2+w_1+w_2)}\\
&&\!\!\!\!\!\!\!\!\!\!\ \frac{L(1+\beta+\gamma_1+w_1+s,\overline{\chi}^2)\zeta_q(1+\alpha+\gamma_1+w_1+s)\zeta_q(1+2w_1)}{L_q(1+\alpha+w_1+s,\chi^2)L(1+\beta+w_1+s,\overline{\chi}^2)\zeta_q(1+\alpha+w_1+s)\zeta(1+\beta+w_1+s)\zeta_q(1+\gamma_1+2w_1)}\\
&&\!\!\!\!\!\!\!\!\!\!\ \ \ \frac{L(1+\alpha+\gamma_2+w_2+s,\chi^2)\zeta_q(1+\beta+\gamma_2+w_2+s)\zeta_q(1+2w_2)}{L(1+\alpha+w_2+s,\chi^2)L_q(1+\beta+w_2+s,\overline{\chi}^2)\zeta(1+\alpha+w_2+s)\zeta_q(1+\beta+w_2+s)\zeta_q(1+\gamma_2+2w_2)}.
\end{eqnarray*}
Here $E(\alpha,\beta,\gamma_1,\gamma_2,w_1,w_2,s)$ is an arithmetical factor given by some Euler product that is absolutely and uniformly convergent in some product of fixed half-planes containing the origin. Again we first move the $w_1$-contour and $w_2$-contour to $\textrm{Re}(w_1)=\textrm{Re}(w_2)=\delta$, and then move the $s$-contour to $\textrm{Re}(s)=-(1-\varepsilon)\delta$, where $\delta,\varepsilon>0$ are some fixed small constants such that the arithmetical factor converges absolutely and $\Delta_2<1-\varepsilon$. In doing so we only cross a simple pole at $s=0$ and the contribution along the new line is $O_\varepsilon(q^{-\varepsilon})$. We denote
\begin{eqnarray*}
&&\!\!\!\!\!\!\!\!\!\!L_2(\alpha,\beta,\gamma_1,\gamma_2)=\bigg(\frac{1}{2\pi i}\bigg)^2\int_{(\delta)}\int_{(\delta)}y_{2}^{w_1+w_2}E(\alpha,\beta,\gamma_1,\gamma_2,w_1,w_2,0)\\
&&\!\!\!\!\!\!\!\!\!\!\quad\frac{L(1+w_1+w_2,\chi^2)L(1+w_1+w_2,\overline{\chi}^2)\zeta(1+\gamma_1+\gamma_2+w_1+w_2)\zeta^2(1+w_1+w_2)}{L(1+\gamma_1+w_1+w_2,\overline{\chi}^2)L(1+\gamma_2+w_1+w_2,\chi^2)\zeta(1+\gamma_1+w_1+w_2)\zeta(1+\gamma_2+w_1+w_2)}\\
&&\ \frac{L(1+\beta+\gamma_1+w_1,\overline{\chi}^2)\zeta_q(1+\alpha+\gamma_1+w_1)\zeta_q(1+2w_1)}{L_q(1+\alpha+w_1,\chi^2)L(1+\beta+w_1,\overline{\chi}^2)\zeta_q(1+\alpha+w_1)\zeta(1+\beta+w_1)\zeta_q(1+\gamma_1+2w_1)}\\
&&\ \ \ \frac{L(1+\alpha+\gamma_2+w_2,\chi^2)\zeta_q(1+\beta+\gamma_2+w_2)\zeta_q(1+2w_2)}{L(1+\alpha+w_2,\chi^2)L_q(1+\beta+w_2,\overline{\chi}^2)\zeta(1+\alpha+w_2)\zeta_q(1+\beta+w_2)\zeta_q(1+\gamma_2+2w_2)}\frac{dw_1}{w_{1}^{i+1}}\frac{dw_2}{w_{2}^{j+1}},
\end{eqnarray*}
so that
\begin{equation}\label{21}
J_{2}^{+}(\alpha,\beta)=\frac{\zeta(1+\alpha+\beta)}{\mathscr{L}^2}\sum_{i,j}\frac{b_ib_ji!j!}{(\log y_2)^{i+j}}\frac{\partial^2}{\partial\gamma_1\partial\gamma_2}L_2(\alpha,\beta,\gamma_1,\gamma_2)\bigg|_{\gamma_1=\gamma_2=0}+O_\varepsilon(q^{-\varepsilon}).
\end{equation}

We now take the derivatives with respect to $\gamma_1,\gamma_2$ and set $\gamma_1=\gamma_2=0$. We first note that by moving the contours to $\textrm{Re}(w_1)=\textrm{Re}(w_2)\asymp \mathscr{L}^{-1}$ and bounding the integral with absolute values, we get $L_2(\alpha,\beta,\gamma_1,\gamma_2)\ll \mathscr{L}^{i+j-1}$. Hence
\begin{equation}\label{302}
\frac{\partial^2}{\partial \gamma_1\partial \gamma_2}L_2(\alpha,\beta,\gamma_1,\gamma_2)\bigg|_{\gamma_1=\gamma_2=0}=L_{21}(\alpha,\beta)+L_{22}(\alpha,\beta)+O(\mathscr{L}^{i+j}),
\end{equation}
where
\begin{eqnarray*}
&&\!\!\!\!\!\!\!\!\!\!\!L_{21}(\alpha,\beta)=\bigg(\frac{1}{2\pi i}\bigg)^2\int\int_{(\mathscr{L}^{-1})}y_{2}^{w_1+w_2}E(\alpha,\beta,0,0,w_1,w_2,0)\zeta(1+w_1+w_2)\\
&&\ \frac{1}{L_q(1+\alpha+w_1,\chi^2)}\bigg(\frac{\zeta_q'(1+\alpha+w_1)}{\zeta_q(1+\alpha+w_1)\zeta(1+\beta+w_1)}-\frac{\zeta_q'(1+2w_1)}{\zeta(1+\beta+w_1)\zeta_q(1+2w_1)}\bigg)\\
&&\ \frac{1}{L_q(1+\beta+w_2,\overline{\chi}^2)}\bigg(\frac{\zeta_q'(1+\beta+w_2)}{\zeta(1+\alpha+w_2)\zeta_q(1+\beta+w_2)}-\frac{\zeta_q'(1+2w_2)}{\zeta(1+\alpha+w_2)\zeta_q(1+2w_2)}\bigg)\frac{dw_1}{dw_{1}^{i+1}}\frac{dw_2}{dw_{2}^{j+1}}
\end{eqnarray*}
and
\begin{eqnarray*}
&&\!\!\!\!\!\!\!\!\!\!\!L_{22}(\alpha,\beta)=\bigg(\frac{1}{2\pi i}\bigg)^2\int\int_{(\mathscr{L}^{-1})}y_{2}^{w_1+w_2}E(\alpha,\beta,0,0,w_1,w_2,0)\bigg(\zeta''(1+w_1+w_2)-\frac{\zeta'(1+w_1+w_2)^2}{\zeta(1+w_1+w_2)}\bigg)\\
&&\qquad\frac{1}{L_q(1+\alpha+w_1,\chi^2)\zeta(1+\beta+w_1)}\frac{1}{L_q(1+\beta+w_2,\overline{\chi}^2)\zeta(1+\alpha+w_2)}\frac{dw_1}{dw_{1}^{i+1}}\frac{dw_2}{dw_{2}^{j+1}}.
\end{eqnarray*}

As in the previous sections, we can replace $L_q(1+\alpha+w_1,\chi^2)L_q(1+\beta+w_2,\overline{\chi}^2)$ and $E(\alpha,\beta,0,0,w_1,w_2,0)$ in the integrals by $|L(1,\chi^2)|^2$ and $E(0,0,0,0,0,0,0)$, respectively, with an admissible error. It is also standard to check that $E(0,0,0,0,0,0,0)=\big(1+O(q^{-1})\big)|L(2,\chi^4)|^2$, a result we will use freely from now on. The new integrals have already been evaluated in [\textbf{\ref{B}}] (see Lemma 6.1 and Lemma 6.2). From there we obtain
\begin{eqnarray*}
L_{21}(\alpha,\beta)&=&\frac{|L(2,\chi^4)|^2(\log y_2)^{i+j+1}}{2|L(1,\chi^2)|^2}\frac{d^2}{dadb}\int_{0}^{1}y_{2}^{\alpha b+\beta a}(1-x)^2\frac{(x+a)^{i}}{i!}\frac{(x+b)^{j}}{j!}dx\bigg|_{a=b=0}\\
&&\qquad\qquad+O(\mathscr{L}^{i+j})+O(\mathscr{L}^{i+\varepsilon})+O(\mathscr{L}^{j+\varepsilon}),
\end{eqnarray*}
and
\begin{eqnarray*}
&&\!\!\!\!\!\!\!\!\!\!\!\!L_{22}(\alpha,\beta)=\frac{|L(2,\chi^4)|^2(\log y_2)^{i+j+1}}{|L(1,\chi^2)|^2}\frac{d^2}{dadb}\bigg\{\int_{0}^{1}\int_{0}^{x}\int_{0}^{x}y_{2}^{\alpha b+\beta a-\alpha u-\beta v}\frac{(x-u+a)^i}{i!}\frac{(x-v+b)^j}{j!}dudvdx\nonumber\\
&&\qquad\qquad-\tfrac{1}{2}\int_{0}^{1}\int_{0}^{x}y_2^{\alpha b+\beta a-\alpha u}\frac{(x-u+a)^i}{i!}\frac{(x+b)^{j+1}}{(j+1)!}dudx\\
&&\qquad\qquad-\tfrac{1}{2}\int_{0}^{1}\int_{0}^{x}y_2^{\alpha a+\beta b-\beta u}\frac{ (x+b)^{i+1}}{(i+1)!}\frac{(x-u+a)^j}{j!}dudx\nonumber\\
&&\qquad\qquad+\tfrac{1}{4}\int_{0}^{1}y_{2}^{\alpha b+\beta a}\frac{(x+a)^{i+1}}{(i+1)!}\frac{ (x+b)^{j+1}}{(j+1)!}dx\bigg\}\bigg|_{a=b=0}+O(\mathscr{L}^{i+j})+O(\mathscr{L}^{i+\varepsilon})+O(\mathscr{L}^{j+\varepsilon}).\nonumber
\end{eqnarray*}
We collect these evaluations, \eqref{302}, \eqref{21} and write $J_{2}^{+}(\alpha,\beta)$ in a compact form as
\begin{eqnarray}\label{402}
J_{2}^{+}(\alpha,\beta)&=&\bigg|\frac{L(2,\chi^4)}{L(1,\chi^2)}\bigg|^2\frac{1}{\mathscr{L}(\alpha+\beta)}\frac{d^2}{dadb}\bigg\{\tfrac{\Delta_2}{2}\int_{0}^{1}y_{2}^{\alpha b+\beta a}(1-x)^2Q(x+a)Q(x+b)dx\nonumber\\
&&\qquad+\Delta_2\int_{0}^{1}\int_{0}^{x}\int_{0}^{x}y_{2}^{\alpha b+\beta a-\alpha u-\beta v}Q(x-u+a)Q(x-v+b)dudvdx\nonumber\\
&&\qquad-\tfrac{\Delta_2}{2}\int_{0}^{1}\int_{0}^{x}y_2^{\alpha b+\beta a-\alpha u}Q(x-u+a)Q_1(x+b)dudx\nonumber\\
&&\qquad-\tfrac{\Delta_2}{2}\int_{0}^{1}\int_{0}^{x}y_2^{\alpha a+\beta b-\beta u}Q(x-u+a)Q_1(x+b)dudx\nonumber\\
&&\qquad+\tfrac{\Delta_2}{4}\int_{0}^{1}y_{2}^{\alpha b+\beta a}Q_1(x+a)Q_1(x+b)dx\bigg\}\bigg|_{a=b=0}+O(\mathscr{L}^{-1}).
\end{eqnarray}

\subsection{Deduction of Lemma \ref{307}}

Next we combine $J_{2}^{+}(\alpha,\beta)$ and $J_{2}^{-}(\alpha,\beta)$. We note that essentially $J_{2}^{-}(\alpha,\beta)=\hat{q}^{-2(\alpha+\beta)}J_{2}^{+}(-\beta,-\alpha)$. Writing
\begin{equation*}
U_2(\alpha,\beta;u,v)=\frac{y_{2}^{\alpha b+\beta a-\alpha u-\beta v}-\hat{q}^{-2(\alpha+\beta)}y_{2}^{-\beta b-\alpha a+\beta u+\alpha v}}{\alpha+\beta}.
\end{equation*}
Using \eqref{224} we have
\begin{equation*}
U_2(\alpha,\beta;u,v)=\mathscr{L}y_{2}^{\alpha b+\beta a-\alpha u-\beta v}\big(2+\Delta_2(a+b-u-v)\big)\int_{0}^{1}(\hat{q}^2y_{2}^{a+b-u-v})^{-(\alpha+\beta)t}dt.
\end{equation*}
In view of \eqref{402} and simplify, we obtain \eqref{303}.

\section{Proof of Theorem \ref{308}}

\subsection{Removing the harmonic weight}

To deduce Theorem \ref{308}, we first need to remove the weights $w_f$ in Lemmas 2.1--2.4 so that the lemmas also hold for the natural average,
\begin{equation*}
{\sum_{f\in S_{2}^{*}(q)}\!\!\!\!}^{n}\ A_f:=\sum_{f\in S_{2}^{*}(q)}\frac{A_f}{|S_{2}^{*}(q)|}.
\end{equation*}
This technique has been done several times (see [\textbf{\ref{KM}},\textbf{\ref{IS}},\textbf{\ref{KMV}},\textbf{\ref{KMV1}}]), so here we shall only illustrate the method for the mollified first moment of $M_2(f,\chi)$:
\begin{equation*}
{\sum_{f\in S_{2}^{*}(q)}\!\!\!\!}^{n}\ L(f.\chi,\tfrac{1}{2}+\alpha)M_2(f,\chi).
\end{equation*}

We borrow a general lemma from [\textbf{\ref{KM}}]:

\begin{lemma}
Let $(A_f)_{f\in S_{2}^{*}(q)}$ be a family of complex numbers satisfying
\begin{eqnarray}
&& {\sum_{f\in S_{2}^{*}(q)}\!\!\!\!}^{h}\ |A_f|\ll\mathscr{L}^B\quad\qquad\ \ \textrm{for some absolute $B>0$};\label{512}\\
&&\emph{Max}_{f\in S_{2}^{*}(q)}w_f|A_f|\ll \hat{q}^{-c}\quad\textrm{for some absolute $c>0$}.\label{513}
\end{eqnarray}
Then for all $\kappa>0$, there exists $\delta=\delta(B,c)>0$ such that
\begin{eqnarray*}
{\sum_{f\in S_{2}^{*}(q)}\!\!\!\!}^{n}\ A_f=\frac{1}{\zeta(2)}{\sum_{f\in S_{2}^{*}(q)}\!\!\!\!}^{h}\ w_f(\hat{q}^\kappa)A_f+O_{\kappa,B,c}(\hat{q}^{-\delta}),
\end{eqnarray*}
where
\begin{equation*}
w_f(\hat{q}^\kappa)=\sum_{lm^2\leq \hat{q}^\kappa}\frac{\lambda_f(l^2)}{lm^2}.
\end{equation*}
\end{lemma}

We shall apply this lemma to $A_f=L(f.\chi,\tfrac{1}{2}+\alpha)M_2(f,\chi)$. Condition \eqref{512} follows immediately from Lemma 2.4 and Cauchy's inequality. For condition \eqref{513}, it is known that $w_f\ll \mathscr{L}/q$ [\textbf{\ref{GHL}}]. Hence \eqref{513} is satisfied using the convexity bound $L(f.\chi,\tfrac{1}{2}+\alpha)\ll \hat{q}^{1/2}$ and the trivial bound $M_2(f,\chi)\ll \hat{q}^{\Delta_2/2}$.

Thus we are left with estimating the sum
\begin{eqnarray*}
I=\frac{1}{\zeta(2)}{\sum_{f\in S_{2}^{*}(q)}\!\!\!\!}^{h}\ w_f(\hat{q}^\kappa)L(f.\chi,\tfrac{1}{2}+\alpha)M_2(f,\chi).
\end{eqnarray*}
Using the expression \eqref{514}, and applying Lemma 3.1 for the product $\lambda_f(m_1n)\lambda_f(l^2)$ we have
\begin{eqnarray*}
I&=&\frac{1}{\zeta(2)\mathscr{L}}\sum_{\substack{d_1l_1=d_2l_2\\d_1d_2=d_3d_4\\d_1l_1m^2\leq \hat{q}^\kappa\\(ud_1d_2d_3d_4,q)=1}}\frac{\mu(ud_3m_1)(\mu*\log)(m_2)\chi^2(u)\chi(d_3d_4m_1n)\overline{\chi}(m_2)}{\psi(ud_3m_1)\overline{\psi}(m_2)(d_4n)^{1/2+\alpha}u^{1+\alpha}\sqrt{d_3m_1m_2}d_1l_1m^2}\\
&&\qquad\qquad\qquad\qquad\lambda_f(m_1nl_1l_2)\lambda_f(m_2)Q[ud_3m_1m_2]V\bigg(\frac{ud_4n}{\hat{q}^{2+\varepsilon}}\bigg)+O_{\varepsilon,B}(\hat{q}^{-B+\Delta_2/2+\varepsilon}).
\end{eqnarray*}
For the off-diagonal terms, $m_2\ne m_1nl_1l_2$, Lemma \ref{102} implies that the total contribution is
\begin{equation*}
\ll_\varepsilon q^{-1/2+\varepsilon}\sum_{\substack{l\leq \hat{q}^{\kappa}\\m_1m_2\leq y_2}}1\ll_\varepsilon q^{-1/2+\Delta_2/2+\kappa/2+\varepsilon}.
\end{equation*} 
We shall choose $\kappa<1-\Delta_2$ so that the above error term is admissible. The main contribution to $I$, which comes from the terms $m_2=m_1nl_1l_2$, is
\begin{equation*}
\frac{1}{\zeta(2)\mathscr{L}}\sum_{\substack{d_1l_1=d_2l_2\\d_1d_2=d_3d_4\\m_2=m_1nl_1l_2\\d_1l_1m^2\leq \hat{q}^\kappa\\(ud_1d_2d_3d_4,q)=1}}\frac{\mu(ud_3m_1)(\mu*\log)(m_2)\chi^2(u)\chi(d_3d_4)\overline{\chi}(l_1l_2)}{\psi(ud_3m_1)\overline{\psi}(m_2)(d_4n)^{1/2+\alpha}u^{1+\alpha}\sqrt{d_3m_1m_2}d_1l_1m^2}Q[ud_3m_1m_2]V\bigg(\frac{ud_4n}{\hat{q}^{2+\varepsilon}}\bigg).
\end{equation*}
Using \eqref{8}, \eqref{9} and Perron's formula we can write this as
\begin{eqnarray*}
I'&=&\frac{1}{\zeta(2)\mathscr{L}}\sum_{j}\frac{b_jj!}{(\log y_2)^j}\bigg(\frac{1}{2\pi i}\bigg)^3\int_{(2)}\int_{(2)}\int_{(2)}e^{s^2}\hat{q}^{(2+\varepsilon)s}y_{2}^{w_1}\hat{q}^{\kappa w_2}\\
&&\qquad\qquad\qquad\sum_{\substack{d_1l_1=d_2l_2\\d_1d_2=d_3d_4\\m_2=m_1nl_1l_2\\(ud_1d_2d_3d_4,q)=1}}\frac{\mu(ud_3m_1)(\mu*\log)(m_2)\chi^2(u)\chi(d_3d_4)\overline{\chi}(l_1l_2)}{\psi(ud_3m_1)\overline{\psi}(m_2)(d_4n)^{1/2+\alpha}u^{1+\alpha}\sqrt{d_3m_1m_2}d_1l_1m^2}\\
&&\qquad\qquad\qquad\qquad\qquad\qquad\qquad\frac{1}{(ud_3m_1m_2)^{w_1}}\frac{1}{(d_1l_1m^2)^{w_2}}\frac{1}{(ud_4n)^s}\frac{dw_1}{w_{1}^{j+1}}\frac{dw_2}{w_2}\frac{ds}{s}.
\end{eqnarray*}
The sum in the integrand is
\begin{eqnarray*}
&&-\frac{d}{d\gamma}\sum_{\substack{d_1l_1=d_2l_2\\d_1d_2=d_3d_4\\m_{21}m_{22}=m_1nl_1l_2\\(ud_1d_2d_3d_4,q)=1}} \frac{\mu(ud_3m_1)\mu(m_{21})\chi^2(u)\chi(d_3d_4)\overline{\chi}(l_1l_2)}{\psi(ud_3m_1)\overline{\psi}(m_{21}m_{22})(d_4n)^{1/2+\alpha}u^{1+\alpha}\sqrt{d_3m_1m_{21}m_{22}}d_1l_1m^2m_{22}^\gamma}\\
&&\qquad\qquad\qquad\qquad\qquad\qquad\qquad\qquad\qquad\frac{1}{(ud_3m_1m_{21}m_{22})^{w_1}}\frac{1}{(d_1l_1m^2)^{w_2}}\frac{1}{(ud_4n)^s}\bigg|_{\gamma=0}.
\end{eqnarray*}
After some standard calculations, the above sum is
\begin{equation}\label{516}
\frac{F(\alpha,\gamma,w_1,w_2,s)\zeta(2+2w_2)\zeta(1+\alpha+\gamma+w_1+s)\zeta(1+2w_1)}{L_q(1+\alpha+w_1+s,\chi^2)\zeta(1+\alpha+w_1+s)\zeta(1+\gamma+2w_1)},
\end{equation}
where $F(\alpha,\gamma,w_1,w_2,s)$ is an arithmetical factor given by some Euler product that is absolutely and uniformly convergent in some product of fixed half-planes containing the origin. We first move the $w_1$-contour and the $w_2$-contour to $\textrm{Re}(w_1)=\delta$ and $\textrm{Re}(w_2)=-\delta$, and then move the $s$-contour to $\textrm{Re}(s)=-2\delta/(2+\varepsilon)$, where $\delta>0$ is some fixed small constant such that the arithmetical factor converges absolutely. In doing so we only cross simple poles at $w_2=0$ and $s=0$. By bounding the integral by absolute values, the contribution along the new line is
\begin{equation*}
\ll_{\delta} \hat{q}^{-2\delta}y_{2}^{\delta}\hat{q}^{\kappa\delta}\ll_{\delta} \hat{q}^{-(2-\kappa-\Delta_2)\delta}.
\end{equation*} 
Hence 
\begin{equation*}
I'=-\frac{1}{\mathscr{L}}\sum_{j}\frac{b_jj!}{(\log y_2)^j}\frac{\partial}{\partial\gamma}I''(\gamma)\bigg|_{\gamma=0},
\end{equation*}
where
\begin{eqnarray*}
I''(\gamma)=\frac{1}{2\pi i}\int_{(\delta)}y_{2}^{w_1}\frac{F(\alpha,\gamma,w_1,0,0)}{L_q(1+\alpha+w_1,\chi^2)}\frac{\zeta(1+\alpha+\gamma+w_1)\zeta(1+2w_1)}{\zeta(1+\alpha+w_1)\zeta(1+\gamma+2w_1)}\frac{dw_1}{w_{1}^{j+1}}+O_{\delta}(\hat{q}^{-(2-\kappa-\Delta_2)\delta}).
\end{eqnarray*}
This is precisely \eqref{515} with $B(\alpha,\gamma,w,0)$ being replaced by $F(\alpha,\gamma,w_1,0,0)$. Thus we are left to check that $F(0,0,0,0,0)=\big(1+O(q^{-1})\big)L(2,\chi^4)$. Taking $\alpha=\gamma=w_2=0$ and $w_1=s$ in \eqref{516} we have
\begin{equation*}
F(0,0,s,0,s)=L_q(1+2s,\chi^2)\sum_{\substack{d_1l_1=d_2l_2\\d_1d_2=d_3d_4\\m_{21}m_{22}=m_1nl_1l_2\\(ud_1d_2d_3d_4,q)=1}} \frac{\mu(ud_3m_1)\mu(m_{21})\chi^2(u)\chi(d_3d_4)\overline{\chi}(l_1l_2)}{\psi(ud_3m_1)\overline{\psi}(m_{21}m_{22})(u^2d_3d_4m_1m_{21}m_{22}n)^{1/2+s}d_1l_1}.
\end{equation*}
Consider the sum over $m_{21}$. The above sum vanishes unless $m_{21}m_{22}=1$, i.e. $m_{21}m_{22}m_1nl_1l_2=1$. Hence
\begin{eqnarray*}
F(0,0,s,0,s)&=&L_q(1+2s,\chi^2)\sum_{\substack{d_3|d^2\\(ud,q)=1}} \frac{\mu(ud_3)\chi^2(u)\chi^2(d)}{\psi(ud_3)u^{1+2s}d^{2+2s}}\\
&=&L_q(1+2s,\chi^2)\sum_{\substack{d_3|d^2\\(d,q)=1}} \frac{\mu(d_3)\chi^2(d)}{\psi(d_3)d^{2+2s}}\sum_{(u,d_3q)=1}\frac{\mu(u)\chi^2(u)}{\psi(u)u^{1+2s}}\\
&=&L_q(1+2s,\chi^2)\sum_{(d,q)=1}\frac{\chi^2(d)}{d^{2+2s}}\sum_{d_3|d^2}\frac{\mu(d_3)}{\psi(d_3)}\prod_{p\nmid d_3q}\bigg(1-\frac{\chi^2(p)}{\psi(p)p^{1+2s}}\bigg)\\
&=&\big(1+O(q^{-1})\big)L_q(2+4s,\chi^4)\sum_{(d,q)=1}\frac{\chi^2(d)}{d^{2+2s}}\prod_{p|d}\bigg\{1-\frac{1}{\psi(p)}\bigg(1-\frac{\chi^2(p)}{\psi(p)p^{1+2s}}\bigg)^{-1}\bigg\}.
\end{eqnarray*}
For $s=0$, the terms vanish unless $d=1$. So $F(0,0,0,0,0)=\big(1+O(q^{-1})\big)L(2,\chi^4)$. That verifies the removal of the harmonic weight for the mollified first moment corresponding to $M_2(f,\chi)$.

\subsection{Deduction of Theorem \ref{308}}

Let
\begin{equation*}
S_{1,k}(M)={\sum_{f\in S_{2}^{*}(q)}\!\!\!\!}^{n}\ \varepsilon_{f.\chi}L^{(k)}(f.\overline{\chi},\tfrac{1}{2})M(f,\chi)
\end{equation*}
and
\begin{equation*}
S_{2,k}(M)={\sum_{f\in S_{2}^{*}(q)}\!\!\!\!}^{n}\ \big|L^{(k)}(f.\chi,\tfrac{1}{2})M(f,\chi)\big|^2.
\end{equation*}
By Cauchy's inequality we have
\begin{equation}\label{510}
\frac{1}{|S_2^{*}(q)|}\sum_{\substack{f\in S_{2}^{*}(q)\\L^{(k)}(f.\chi,\frac{1}{2})\ne0}}1\geq \frac{\big|S_{1,k}(M)\big|^2}{S_{2,k}(M)}.
\end{equation}

The functional equation gives
\begin{equation*}
\varepsilon_{f.\chi}L(f.\overline{\chi},\tfrac{1}{2}+\alpha)=\hat{q}^{-2\alpha}\frac{\Gamma(1-\alpha)}{\Gamma(1+\alpha)}L(f.\chi,\tfrac{1}{2}-\alpha).
\end{equation*}
Hence in view of the above subsection and Lemma \ref{310}, we get
\begin{eqnarray*}
&&\!\!\!\!\!\!\!\!\!\!\!\!\!\!{\sum_{f\in S_{2}^{*}(q)}\!\!\!}^{n}\ \varepsilon_{f.\chi}L(f.\overline{\chi},\tfrac{1}{2}+\alpha)M(f,\chi)=\frac{L(2,\chi^4)}{L(1,\chi^2)}\hat{q}^{-2\alpha}\frac{\Gamma(1-\alpha)}{\Gamma(1+\alpha)}\\
&&\qquad\qquad\qquad\qquad\qquad\qquad\bigg(P(1)+\Delta_2\int_{0}^{1}y_{2}^{\alpha(1-x)}Q(x)dx-\tfrac{\Delta_2}{2}Q_{1}(1)\bigg)+O(\mathscr{L}^{-1}).
\end{eqnarray*}
Thus, using Cauchy's theorem,
\begin{equation*}
S_{1,k}(M)=(-1)^k\frac{L(2,\chi^4)}{L(1,\chi^2)}\bigg(2^kP(1)+\Delta_2\int_{0}^{1}\big(2-\Delta_2(1-x)\big)^kQ(x)dx-2^{k-1}\Delta_2Q_1(1)\bigg)\mathscr{L}^k +O_k(\mathscr{L}^{k-1}).
\end{equation*}
Similarly,
\begin{eqnarray*}
S_{2,k}(M)&=&\bigg|\frac{L(2,\chi^4)}{L(1,\chi^2)}\bigg|^2\frac{d^2}{dadb}\bigg\{\int_{0}^{1}\int_{0}^{1}\big((2+\Delta_1(a+b))t-\Delta_1a\big)^k\big((2+\Delta_1(a+b))t-\Delta_1b\big)^k\\
&&\qquad\qquad\qquad\qquad\qquad\qquad\qquad\qquad\big(2\Delta_{1}^{-1}+a+b\big)P(x+a)P(x+b)dxdt\\
&&\!\!\!\!\!\!\!\!\!\!\!\!\!\!\!\!\!\!\!\!\!\!\!\!\!\!\!\!\!\!\!\!\!\!+\tfrac{2\Delta_2}{\Delta_1}\int_{0}^{1}\int_{0}^{1}\int_{0}^{x}\big((2+\Delta_1a+\Delta_2(b-u))t-\Delta_1a+\Delta_2u\big)^k\big((2+\Delta_1a+\Delta_2(b-u))t-\Delta_2b\big)^k\\
&&\qquad\qquad\qquad\big(2+\Delta_1a+\Delta_2(b-u)\big)P\big(1-\tfrac{\Delta_2(1-x)}{\Delta_1}+a\big)Q(x-u+b)dudxdt\\
&&\!\!\!\!\!\!\!\!\!\!\!\!\!\!\!\!\!\!\!\!\!\!\!\!\!\!\!\!\!\!\!\!\!\!-\tfrac{\Delta_2}{\Delta_1}\int_{0}^{1}\int_{0}^{1}\big((2+\Delta_1a+\Delta_2b)t-\Delta_1a\big)^k\big((2+\Delta_1a+\Delta_2b)t-\Delta_2b\big)^k\\
&&\qquad\qquad\qquad \big(2+\Delta_1a+\Delta_2b\big)P\big(1-\tfrac{\Delta_2(1-x)}{\Delta_1}+a\big)Q_1(x+b)dxdt\\
&&\!\!\!\!\!\!\!\!\!\!\!\!\!\!\!\!\!\!\!\!\!\!\!\!\!\!\!\!\!\!\!\!\!\!+\tfrac{\Delta_2}{2}\int_{0}^{1}\int_{0}^{1}\big((2+\Delta_2(a+b))t-\Delta_2a\big)^k\big((2+\Delta_2(a+b))t-\Delta_2b\big)^k\\
&&\qquad\qquad\qquad\big(2+\Delta_2(a+b)\big)(1-x)^2Q(x+a)Q(x+b)dxdt\nonumber\\
&&\!\!\!\!\!\!\!\!\!\!\!\!\!\!\!\!\!\!\!\!\!\!\!\!\!\!\!\!\!\!\!\!\!\!+\Delta_2\int_{0}^{1}\int_{0}^{1}\int_{0}^{x}\int_{0}^{x}\big((2+\Delta_2(a+b-u-v))t-\Delta_2(a-v)\big)^k\big((2+\Delta_2(a+b-u-v))t-\Delta_2(b-u)\big)^k\nonumber\\
&&\qquad\qquad\qquad\big(2+\Delta_2(a+b-u-v)\big)Q(x-u+a)Q(x-v+b)dudvdxdt\nonumber\\
&&\!\!\!\!\!\!\!\!\!\!\!\!\!\!\!\!\!\!\!\!\!\!\!\!\!\!\!\!\!\!\!\!\!\!-\Delta_2\int_{0}^{1}\int_{0}^{1}\int_{0}^{x}\big((2+\Delta_2(a+b-u))t-\Delta_2a\big)^k\big((2+\Delta_2(a+b-u))t-\Delta_2(b-u)\big)^k\\
&&\qquad\qquad\qquad\big(2+\Delta_2(a+b-u)\big)Q(x-u+a)Q_1(x+b)dudxdt\nonumber\\
&&\!\!\!\!\!\!\!\!\!\!\!\!\!\!\!\!\!\!\!\!\!\!\!\!\!\!\!\!\!\!\!\!\!\!+\tfrac{\Delta_2}{4}\int_{0}^{1}\int_{0}^{1}\big((2+\Delta_2(a+b))t-\Delta_2a\big)^k\big((2+\Delta_2(a+b))t-\Delta_2b\big)^k\\
&&\qquad\qquad\qquad\big(2+\Delta_2(a+b)\big)Q_1(x+a)Q_1(x+b)dxdt\bigg\}\bigg|_{a=b=0}\mathscr{L}^{2k} +O_k(\mathscr{L}^{2k-1}).
\end{eqnarray*}

Specific values for $p_{k,\chi}$, for small $k$, are calculated with Mathematica. The results are summarised in the table below.
\begin{table}[ht]
\caption{The lower bounds for the proportions $p_{k,\chi}$ in the table are obtained by using the inequality \eqref{510} and the expressions for $S_{1,k}(M)$ and $S_{2,k}(M)$ given above with $\Delta_1=\Delta_2=1$.}\label{eqtable}
\renewcommand\arraystretch{1.5}
\noindent\[
\begin{array}{|c|c|c|c|}
\hline
\ k \ & P(x) \ & Q(x) \ & \text{Lower bound for}\ p_{k,\chi} \ \\
\hline
$0$ & 1.05x-0.05x^2 & 0.9x &0.3411\\
$1$ & 0.87x+0.13x^3 & 0.15x - 0.11x^2 &0.7553\\
$2$ & 0.75x+0.25x^3 & 0.06x-0.05x^2 &0.9085\\
$3$ & 0.62x+0.32x^3+0.06x^5 & 0.03x-0.04x^2 &0.9643\\
\hline
\end{array}
\]
\end{table}

For general $k$ $(k\geq4)$, we take $Q(x)=0$ and obtain
\begin{equation*}
S_{1,k}(M)=(-1)^k\frac{L(2,\chi^4)}{L(1,\chi^2)}2^kP(1)\mathscr{L}^k+O_k(\mathscr{L}^{k-1})
\end{equation*}
and
\begin{equation*}
S_{2,k}(M)=\bigg|\frac{L(2,\chi^4)}{L(1,\chi^2)}\bigg|^2\bigg\{\frac{2^{2k+1}}{2k+1}\int_{0}^{1}P'(x)^2dx+2^{2k-1}P(1)^2+\frac{2^{2k-1}k^2}{2k-1}\int_{0}^{1}P(x)^2dx\bigg\}\mathscr{L}^{2k}+O_k(\mathscr{L}^{2k-1}).
\end{equation*}
Hence
\begin{equation*}
p_{k,\chi}\geq\bigg\{\frac{2}{2k+1}\int_{0}^{1}P'(x)^2dx+\frac{1}{2}+\frac{k^2}{2(2k-1)}\int_{0}^{1}P(x)^2dx\bigg\}^{-1}.
\end{equation*}
We need to maximise the expression on the right hand side under the conditions that $P(0)=0$ and $P(1)=1$. This optimisation problem has been solved explicitly using the calculus of variations in [\textbf{\ref{MV}},\textbf{\ref{BM}}]. From that we get
\begin{equation*}
p_{k,\chi}\geq\bigg\{\frac{1}{2}+\frac{k}{\sqrt{4k^2-1}}\coth\bigg[\frac{k}{2}\sqrt{\frac{2k+1}{2k-1}}\bigg]\bigg\}^{-1},
\end{equation*} 
and hence
\begin{equation*}
p_{k,\chi}\geq 1-\frac{1}{16k^2}+O(k^{-4}).
\end{equation*}
This completes the proof of the theorem.

\section{Proof of Theorem 1.2}

Our theorem is similar to Theorem 1.4 of Kowalski, Michel and VanderKam [\textbf{\ref{KMV}}]. We give a sketch of the proof following their Section 8, and refer the readers to [\textbf{\ref{KMV}},\textbf{\ref{KM}}] for complete details.

We shall need an upper bound for the average rank squared.

\begin{proposition}\label{501}
There exists an absolute constant $C>0$ such that for all $q$ prime
\begin{equation*}
{\sum_{f\in S_{2}^{*}(q)}\!\!\!\!}^{n}\ r_{f.\chi}^{2}\leq C.
\end{equation*}
\end{proposition}

We now deduce Theorem 1.2 and postpone the proof of Proposition 10.1 until later. Let $0<m<2$ be fixed. We consider the sum
\begin{equation*}
T(m)=\sum_{f\in S_{2}^{*}(q)}r_{f.\chi}^{m}.
\end{equation*} 
We write $T(m)=T_1+T_2$ where
\begin{equation*}
T_1=\sum_{\substack{f\in S_{2}^{*}(q)\\r_{f.\chi}>n}}r_{f.\chi}^{m}\quad\textrm{and}\quad T_2=\sum_{\substack{f\in S_{2}^{*}(q)\\r_{f.\chi}\leq n}}r_{f.\chi}^{m}.
\end{equation*}
By H\"older's inequality we have
\begin{eqnarray*}
T_1\leq \bigg(\sum_{f}r_{f.\chi}^{2}\bigg)^{m/2}\bigg(\sum_{r_{f.\chi}>n}1\bigg)^{1-m/2}.
\end{eqnarray*}
If the analytic rank $r_{f.\chi}>n$, then $L^{(n)}(f.\chi,1/2)=0$. The proportion of the forms satisfying this condition is
\begin{equation}\label{502}
\leq 1-p_{n,\chi}+o_n(1).
\end{equation} 
Combining with Proposition \ref{501} we obtain
\begin{eqnarray*}
T_1&\ll& |S_{2}^{*}(q)|^{m/2}\bigg(\big(1-p_{n,\chi}+o_n(1)\big)|S_{2}^{*}(q)|\bigg)^{1-m/2}\\
&\ll&\big(n^{-2+m}+o_n(1)\big)|S_{2}^{*}(q)|.
\end{eqnarray*}
To estimate $T_2$, partial summation yields
\begin{equation*}
T_2=\sum_{k=1}^{n}k^m\bigg(\sum_{r_{f.\chi}=k}1\bigg)\leq \sum_{k=1}^{n}\big(k^m-(k-1)^m\big)\bigg(\sum_{r_{f.\chi}\geq k}1\bigg).
\end{equation*}
Hence in view of \eqref{502},
\begin{eqnarray*}
T_2\leq \bigg(\sum_{k=0}^{n-1}\big((k+1)^m-k^m\big)(1-p_{k,\chi})+o_n(1)\bigg)|S_{2}^{*}(q)|.
\end{eqnarray*}
We extend the sum to the infinite series and add the contribution of $T_1$ to obtain
\begin{equation*}
T(m)\leq \bigg(\sum_{k=0}^{\infty}\big((k+1)^m-k^m\big)(1-p_{k,\chi})+o(1)\bigg)|S_{2}^{*}(q)|.
\end{equation*}
In particular when $m=1$, using Theorem 1.1 and the fact that for $k\geq4$,
\begin{equation*}
p_{k,\chi}\geq \bigg\{\frac{1}{2}+\frac{k}{\sqrt{4k^2-1}}\coth\bigg[\frac{k}{2}\sqrt{\frac{2k+1}{2k-1}}\bigg]\bigg\}^{-1}
\end{equation*}
we deduce that
\begin{equation*}
\sum_{f\in S_{2}^{*}(q)}r_{f.\chi}\leq \big(1.0656+o(1)\big)|S_{2}^{*}(q)|,
\end{equation*}
and that proves Theorem 1.2.

\begin{proof}[Proof of Proposition \ref{501}]

Applying the explicit formula as in [\textbf{\ref{B1}}] and proceeding in the same way as [\textbf{\ref{KM}}] (Section 8), the proof is essentially reduced to a density theorem for zeros of the considered $L$-functions.

For any $\sigma\geq\tfrac{1}{2}$, $t_1$ and $t_2$ real, let $N(f.\chi;\sigma,t_1,t_2)$ be the number of zeros $\rho=\beta+i\gamma$ of $L(f.\chi,s)$ which satisfy $\beta\geq\sigma$ and $t_1\leq\gamma\leq t_2$. Then it suffices to show that

\begin{lemma}
There exist absolute constants $B,c>0$ such that for any $\sigma\geq\tfrac{1}{2}+\mathscr{L}^{-1}$ and $t_2-t_1\geq \mathscr{L}^{-1}$, we have
\begin{equation*}
\sum_{f\in S_{2}^{*}(q)}N(f.\chi;\sigma,t_1,t_2)^2\ll_{c} \big(1+|t_1|+|t_2|\big)^B\hat{q}^{2-c(\sigma-1/2)}\mathscr{L}^2.
\end{equation*}
\end{lemma}

\begin{remark}
\emph{As noted in [\textbf{\ref{KM}}] (see Section 4.1), we are interested in the $q$-aspect of the density theorem. A polynomial bound with respect to $t_1$ and $t_2$ is sufficient for our purpose.} 
\end{remark}

To prove Lemma 10.1, we shall appeal to a result similar to Proposition 4 of [\textbf{\ref{KM}}], which estimates a mollified second moment of $L(f.\chi,s)$ on average.

Let $y=\hat{q}^{\Delta}$ and let 
\begin{displaymath}
g(x)=\left\{ \begin{array}{ll}
1 &\qquad \textrm{if $x\leq\sqrt{y}$}\\
\frac{2\log y/x}{\log y} & \qquad\textrm{if $\sqrt{y}\leq x\leq y$}\\
0 &\qquad \textrm{if $x>y$}. 
\end{array} \right.
\end{displaymath}
We define the function $M(f.\chi,s)$ to be
\begin{equation*}
M(f.\chi,s)=\sum_{\substack{m,n\geq1\\(n,q)=1}}\frac{\mu(m)\mu(mn)^2\chi(mn^2)\lambda_f(m)g(mn)}{(mn^2)^s}.
\end{equation*}

\begin{lemma}
Let $0<\Delta<\tfrac{1}{2}$. There exist absolute constants $B,c>0$ such that for all $q$ prime large enough we have
\begin{equation*}
\sum_{f\in S_{2}^{*}(q)}\big|L(f.\chi,\sigma+it)M(f.\chi,\sigma+it)-1\big|^2\ll_{c,\Delta}(1+|t|)^B\hat{q}^{2-c(\sigma-1/2)},
\end{equation*}
uniformly for $\sigma\geq\tfrac{1}{2}+\mathscr{L}^{-1}$ and $t\in\mathbb{R}$.
\end{lemma}

\begin{remark}
\emph{The function $M(f.\chi,s)$ defined above is a mollifier for $L(f.\chi,s)$. We note that for $\sigma>1$, the inverse $L(f.\chi,s)^{-1}$ is given by the Dirichlet series}
\begin{equation*}
L(f.\chi,s)^{-1}=\sum_{\substack{m,n\geq1\\(n,q)=1}}\frac{\mu(m)\mu(mn)^2\chi(mn^2)\lambda_f(m)}{(mn^2)^s}\qquad(\sigma>1).
\end{equation*}
\end{remark}

We shall now prove Lemma 10.1 by following an argument of Selberg [\textbf{\ref{S1}}]. We can assume that $t_2-t_1=\mathscr{L}^{-1}$. Let 
\begin{equation*}
\sigma'=\sigma-1/2\mathscr{L},\qquad t_1'=t_1-\eta/\mathscr{L}\qquad\textrm{and}\qquad t_2'=t_2+\eta/\mathscr{L},
\end{equation*}
where $\eta>0$ is large enough. Consider the function $h_{f.\chi}(s)=L(f.\chi,s)M(f.\chi,s)$, which vanishes at zeros of $L(f.\chi,s)$. Using Lemma 14 and Theorem 4 of Selberg [\textbf{\ref{S1}}] we have
\begin{eqnarray*}
N(f.\chi;\sigma,t_1,t_2)&\leq&\frac{2\mathscr{L}}{\pi}\int_{t_1'}^{t_2'}\sin\bigg(\frac{t-t_1'}{t_2'-t_1'}\pi\bigg)\log\big|h_{f.\chi}(\sigma'+it)\big|dt\\
&&\ +\frac{2\mathscr{L}}{\pi}\int_{\sigma'}^{\infty}\sinh\bigg(\frac{x-\sigma'}{t_2'-t_1'}\pi\bigg)\bigg(\log\big|h_{f.\chi}(x+it_1')\big|+\log\big|h_{f.\chi}(x+it_2')\big|\bigg)dx.
\end{eqnarray*}
We write $h_{f.\chi}(s)=1+\big(LM(f.\chi,s)-1\big)$ and use the inequalities
\begin{equation*}
\log|1+x|\leq|x|,\qquad \sinh(x)\geq0\quad(x>0).
\end{equation*}
Hence
\begin{eqnarray*}
N(f.\chi;\sigma,t_1,t_2)&\leq&\frac{2\mathscr{L}}{\pi}\int_{t_1'}^{t_2'}\sin\bigg(\frac{t-t_1'}{t_2'-t_1'}\pi\bigg)\big|LM(f.\chi,\sigma'+it)-1\big|dt\\
&&\!\!\!\!\!\!\!\!\!\!\!\!\!\!\!\!\!\!\!\!\!\!\!\!\!\!\!\!\!\! +\frac{2\mathscr{L}}{\pi}\int_{\sigma'}^{\infty}\sinh\bigg(\frac{x-\sigma'}{t_2'-t_1'}\pi\bigg)\bigg(\big|LM(f.\chi,x+it_1')-1\big|+\big|LM(f.\chi,x+it_2')-1\big|\bigg)dx.
\end{eqnarray*}
We now square this estimate and average over $f$. The first square term is
\begin{equation*}
\ll \mathscr{L}^2\int_{t_1'}^{t_2'}\int_{t_1'}^{t_2'}\sin\bigg(\frac{t-t_1'}{t_2'-t_1'}\pi\bigg)\sin\bigg(\frac{\tau-t_1'}{t_2'-t_1'}\pi\bigg)\mathscr{M}_1(t,\tau)dtd\tau,
\end{equation*}
where
\begin{equation*}
\mathscr{M}_1(t,\tau)=\sum_{f\in S_{2}^{*}(q)}|LM(f.\chi,\sigma'+it)-1\big||LM(f.\chi,\sigma'+i\tau)-1\big|.
\end{equation*}
By Cauchy's inequality and Lemma 10.2,
\begin{eqnarray*}
\mathscr{M}_1(t,\tau)&\leq&\bigg(\sum_{f\in S_{2}^{*}(q)}|LM(f.\chi,\sigma'+it)-1\big|^2\bigg)^{1/2}\bigg(\sum_{f\in S_{2}^{*}(q)}|LM(f.\chi,\sigma'+i\tau)-1\big|^2\bigg)^{1/2}\\
&\ll_c&\big(1+|t|\big)^B\big(1+|\tau|\big)^B\hat{q}^{2-c(\sigma'-1/2)}.
\end{eqnarray*}
Hence  the first square term is bounded by
\begin{eqnarray*}
&\ll_c& \mathscr{L}^2\hat{q}^{2-c(\sigma'-1/2)}\bigg(\int_{t_1'}^{t_2'}\sin\bigg(\frac{t-t_1'}{t_2'-t_1'}\pi\bigg)\big(1+|t|\big)^Bdt\bigg)^2\\
&\ll_c& \mathscr{L}^2\big(1+|t_1'|+|t_2'|\big)^{2B+2}\hat{q}^{2-c(\sigma-1/2)}(t_2'-t_1')^2.
\end{eqnarray*}
Similarly for the other square terms and we obtain Lemma 10.1. 
\end{proof}

\end{document}